%% file: nthOrderChoquardClassification-arxiv.tex
\documentclass[reqno]{amsart}

\input{structure.tex}

\usepackage{amsaddr}

\newif\ifdetails
\title[$n^{\text{th}}$ order nonlocal classification]{Classification of solutions to an $n^{\text{th}}$ order conformally invariant elliptic equation on $\bb R^n$ with nonlocal nonlinearity}
\author[M. Gluck and J. Marasinghe]{Mathew Gluck and Janani Marasinghe}
\thanks{This material is based upon work supported by the National Science Foundation under Grant No. DMS-2418889.}
\address{Southern Illinois University, \\School of Mathematical and Statistical Sciences\\ Carbondale, IL}
\email{mathew.gluck@siu.edu}
\date{\today}
\begin{document}
\begin{abstract}
This paper concerns a conformally invariant elliptic problem on $\bb R^n$ driven by $(-\lap)^{n/2}$ that has a nonlocal exponential nonlinearity of Choquard type. The problem under consideration is a nonlocal generalization of the constant $Q$-curvature problem on $\bb R^n$. We classify the asymptotic behavior at infinity of all solutions that satisfy a suitable integrability assumption. Under a growth restriction at infinity and a lower bound on the energy we provide an explicit classification for solutions. The classification is heuristically consistent with the classification of the corresponding local problem. 
\end{abstract}
\maketitle
\section{Introduction}
This note concerns the properties of distributional solutions to the problem 
\begin{equation}
\label{eq:entire_nonlocal}
\begin{cases}
	\ds(-\lap)^{n/2} u = (n - 1)!I_\mu[e^{\varrho u}]e^{\varrho u} & \text{ in }\bb R^n\\
	e^u\in L^n(\bb R^n)
\end{cases}
\end{equation}
where $\mu\in (0, n)$, 
\begin{equation}
\label{eq:rho}
	\varrho = n - \frac\mu 2 \in \left(\frac n2, n\right)
\end{equation}
and $I_\mu$ is the convolution operator
\begin{equation*}
	I_\mu[f](x) = \int_{\bb R^n}\frac{f(y)}{|x- y|^\mu}\; \d y. 
\end{equation*}
The motivation for considering problem \eqref{eq:entire_nonlocal} comes from the applications of the problem
\begin{equation}
\label{eq:entire_local}
\begin{cases}
	(-\lap)^{n/2} u = (n - 1)!e^{nu} & \text{ in }\bb R^n\\
	V_0= \int_{\bb R^n}e^{nu} \; \d x < \infty 
\end{cases}
\end{equation}
in conformal geometry and from the heuristic principle that one recovers problem \eqref{eq:entire_local} from problem \eqref{eq:entire_nonlocal} as $\mu\to 0^+$ (after addition of a suitable constant to $u$). A solution $u$ to \eqref{eq:entire_local} corresponds to the metric $g_u = e^{2u}g_0$ on $\bb R^n$ which is conformally equivalent to the standard Euclidean metric $g_0$ and for which the $Q$-curvature, volume, and total $Q$-curvature of $g_u$ are $(n - 1)!$, $V_0$, and $(n - 1)!\int_{\bb R^n}e^{nu}$ respectively. By pulling the standard metric of $\bb S^n$ back to $\bb R^n$ via stereographic projection, one can verify that the function $U(x) = \log 2 - \log(1+ \abs x^2)$ is a solution to \eqref{eq:entire_local} with $V_0 = \abs{\bb S^n}$. Since \eqref{eq:entire_local} is invariant under translations and dilations, each member of the family of functions 
\begin{equation}
\label{eq:SphericalSolutionFamily}
	U_{\bar x,d}(x)
	= 
	\log \frac{2d}{d^2 +\abs{x- \bar x}^2}
\end{equation}
parameterized by $(\bar x, d)\in \bb R^n\times(0, \infty)$ is also a solution to \eqref{eq:entire_local} with $V_0 = \abs{\bb S^n}$. The functions in this family are called \emph{spherical solutions}. 

Over the past three decades, considerable effort has been devoted to classifying all solutions to problem \eqref{eq:entire_local}. Such a classification has applications in proving quantization, compactness and existence results for solutions to equations locally modeled by \eqref{eq:entire_local}, see for example \cite{BrezisMerle1991, LiShafrir1994, RobertStruwe2004, AdimurthiRobertStruwe2006, Malchiodi2006, Robert2007, Martinazzi2009concentration, MartinazziPetrache2010, Martinazzi2011} and the references therein. When $n = 2$  it was shown in \cite{ChenLi1991} that every solution to \eqref{eq:entire_local} is spherical and hence satisfies $V_0= |\bb S^2|$. On the other hand, in higher dimensions there are non spherical solutions to \eqref{eq:entire_local}, see \cite{ChangChen2001}. Concerning general dimension $n\geq 2$, the following asymptotic estimate has been established. 
\begin{oldtheorem}[\cite{Lin1998, Martinazzi2009, Jin2015, Hyder2019}]
\label{OldTheorem:Asymptotic}
Every solution $u$ to \eqref{eq:entire_local} satisfies the asymptotic estimate
\begin{equation}
\label{eq:GeneralDecay}
	u(x)= p(x) - \frac{2V_0}{\abs{\bb S^n}}\log \abs x + \circ(\log \abs x)
	\qquad
	\text{ as }\abs x\to \infty, 
\end{equation}
for some polynomial $p$ that is bounded above and for which ${\rm deg}(p)\leq n -1$.
\end{oldtheorem}
It was also shown in the above cited works that the asymptotic growth assumption $u(x) = \circ(\abs x^2)$ as $\abs x\to \infty$ is necessary and sufficient for a solution to \eqref{eq:entire_local} to be spherical. Partial converses to the asymptotic estimate \eqref{eq:GeneralDecay} are available. It was shown that in even dimensions $n\geq 4$, given any $V_0\in (0, \abs{\bb S^n})$ and any polynomial $p$ satisfying both ${\rm deg}(p)\leq n-2$ and $x\cdot\Grad p(x) \to \infty$ as $\abs x\to \infty$, there is a solution $u$ to \eqref{eq:entire_local} with  asymptotic behavior as in \eqref{eq:GeneralDecay}, see \cite{HyderMartinazzi2015, WeiYe2008}. In odd dimensions $n\geq 3$, it was shown that for every $V_0\in (0, \abs{\bb S^n})$ and every polynomial $p$ satisfying both ${\rm deg}(p)\leq n - 1$ and $p(x)\to-\infty$ there is a solution $u$ to \eqref{eq:entire_local}  satisfying the asymptotic estimate \eqref{eq:GeneralDecay}, see \cite{Jin2015, Hyder2016}. Concerning the possibility of existence of ``large volume'' solutions to \eqref{eq:entire_local} it was shown in \cite{Lin1998} for $n = 4$ and in \cite{Jin2015} for $n = 3$ that every solution to \eqref{eq:entire_local} must satisfy $0< V_0\leq \abs{\bb S^n}$. Thus, in these cases problem \eqref{eq:entire_local} does not admit large volume solutions. On the other hand, it has been shown that in dimensions $n \geq 5$, for every $V_0\in (0, \infty)$ there is a solution to \eqref{eq:entire_local}, see \cite{Martinazzi2013, HuangYe2015, Hyder2017}.

Problems \eqref{eq:entire_local} and \eqref{eq:entire_nonlocal} enjoy invariance under the same symmetries. In particular, for any $(x_0, \sigma)\in \bb R^n\times(0, \infty)$ they are both invariant under the rescaling
\begin{equation}
\label{eq:natural_rescaling}
	u\mapsto u(\sigma(\cdot - x_0)) + \log \sigma, 
\end{equation}
and they are both invariant under the logarithmic Kelvin transform
\begin{equation*}
	u(x)\mapsto u\left(x_0 + \frac{\sigma^2(x - x_0)}{|x - x_0|^2}\right) + 2\log\frac{\sigma}{|x - x_0|}. 
\end{equation*}
These shared symmetries prompt us to pose the following question:
\begin{question}
\label{question:extend_to_nonlocal}
To what extent do the results stated above for problem \eqref{eq:entire_local} hold for problem \eqref{eq:entire_nonlocal}? 
\end{question}
To answer this question we note that the integrability condition in \eqref{eq:entire_nonlocal} together with the Hardy-Littlewood-Sobolev inequality guarantees that the quantity 
\begin{equation}
\label{eq:nonlocal_volume}
	V = \int_{\bb R^n}I_\mu[e^{\varrho u}]e^{\varrho u}
\end{equation}
is finite. As we will see, this quantity plays the role for problem \eqref{eq:entire_nonlocal} that $V_0$ plays for problem \eqref{eq:entire_local}. 
Partial answers to Question \ref{question:extend_to_nonlocal} have already been addressed in the literature. For example, in the case $n = 2$, it was shown in \cite{Gluck2025classification} that if $u$ is a solution to \eqref{eq:entire_nonlocal} then there is $(\bar x, d)\in \bb R^2\times (0, \infty)$ such that 
\begin{equation*}
	u(x) = \log\frac{2d}{d^2 + |x - \bar x|^2} - \log 2+ \frac{1}{4 - \mu}\log\left(\frac{2(2 - \mu)}{\pi}\right)
\end{equation*}
and
\begin{equation*}
\begin{split}
	I_\mu[e^{\varrho u}](x)
	& = \mc H(2, \mu)^{1/\varrho}(4\pi)^{\frac{2 - \mu}{4 - \mu}}e^{\mu u(x)/2}\\
	& = \left(\frac{4}{\pi(2 - \mu)}\right)^{1/2}\left(\frac{2d}{d^2 + |x - \bar x|^2}\right)^{\mu/2},
\end{split} 
\end{equation*}
where $\mc H(2, \mu)$ is the sharp constant in the Hardy-Littlewood-Sobolev inequality with exponents $4/(4 - \mu)$ and $4/\mu$, see \eqref{eq:sharp_HLS_constant} below. In particular, up to an additive constant, every such $u$ is a spherical solution to \eqref{eq:entire_local}. Moreover, the additive constant must be chosen so that both 
\begin{equation*}
	V = |\bb S^2|
	\quad \text{ and }\quad
	\int_{\bb R^2}e^{2u} = \left(4(2 - \mu)\right)^{\frac{2}{2 - \mu}}\left(\frac{\pi^{2 - \mu}}{4}\right)^{\frac1{4 - \mu}}.  
\end{equation*}
In this sense we say that all solutions to \eqref{eq:entire_nonlocal} in the two-dimensional setting are spherical. This classification result was used in \cite{Gluck2025quantization} to prove a quantization result for blowing up sequences of solutions to a problem having problem \eqref{eq:entire_nonlocal} with $n = 2$ as its ``limiting problem''.  

Our aim in this note is to provide further insight to Question \ref{question:extend_to_nonlocal}. Our first result is an analog of Theorem \ref{OldTheorem:Asymptotic} for problem \eqref{eq:entire_nonlocal}. 
\begin{theorem}
\label{theorem:asymptotic}
Let $n\geq 1$, let $\mu\in (0, n)$ and let $\varrho$ be as in \eqref{eq:rho}. If $u$ is a distributional solution to \eqref{eq:entire_nonlocal} then there is a polynomial $p$ of degree at most $n - 1$ that is bounded above for which 
\begin{equation*}
	u(x) = p(x) - \frac{2V}{|\bb S^n|}\log|x| + \circ(\log|x|)
\end{equation*}	
as $|x|\to\infty$. 
\end{theorem}
If, in addition to the assumptions of Theorem \ref{theorem:asymptotic}, $u$ is assumed to satisfy $u(x) = \circ(|x|^2)$ as $|x|\to\infty$, then the polynomial $p$ in the conclusion of Theorem \ref{theorem:asymptotic} is constant. In this case, under a mild lower bound assumption on $V$ we can deduce the explicit form of $u$.   
\begin{theorem}
\label{theorem:classification}
Let $n\geq 1$, let $\mu\in (0, n)$ and let $\varrho$ be as in \eqref{eq:rho}. If $u$ is a distributional solution to \eqref{eq:entire_nonlocal} for which both
\begin{equation}
\label{eq:V_lower_bound_assumption}
	\frac{n - \mu}{2n - \mu}<\frac{V}{|\bb S^n|}
\end{equation}
and 
\begin{equation}
\label{eq:asymtpotic_growth_assumption}
    u(x)= \circ(|x|^2)
    \quad \text{ as }|x|\to\infty, 
\end{equation}
then there is $(\bar x, d)\in \bb R^n\times (0, \infty)$ for which 
\begin{equation}
\label{eq:nonlocal_bubbles}
	u(x) = \log\frac{2d}{d^2 + |x - \bar x|^2} 
	-\frac{1}{2n - \mu}\log\left(\mc H(n, \mu)|\bb S^n|^{1 - \mu/n}\right),
\end{equation}
where $\mc H(n, \mu)$ is the sharp constant in the Hardy-Littlewood-Sobolev inequality with exponents $n/\varrho$ and $2n/\mu$ as given in \eqref{eq:sharp_HLS_constant} below. Moreover, 
\begin{equation}
\label{eq:Imu_erhou_form}
	I_\mu[e^{\varrho u}]
	= \mc H(n, \mu)^{\frac{n}{2n - \mu}}|\bb S^n|^{\frac{n - \mu}{2n - \mu}}e^{\mu u/2}
	= \mc H(n, \mu)^{\frac 12}|\bb S^n|^{\frac{n - \mu}{2n}}e^{\mu U_{\bar x, d}/2}. 
\end{equation}
\end{theorem}
As mentioned above, the asymptotic growth assumption $u(x) = \circ(\abs x^2)$ as $\abs x\to \infty$ is necessary and sufficient for a solution to problem \eqref{eq:entire_local} to be spherical. Theorem \ref{theorem:classification} implies a similar result for solutions to the nonlocal problem \eqref{eq:entire_nonlocal}. Indeed, when Theorem \ref{theorem:classification} applies the equality 
\begin{equation*}
	I_\mu[e^{\varrho u}](x)e^{\varrho u}(x)
	= \left(\frac{2d}{d^2 + |x - \bar x|^2}\right)^n 
	= e^{n U_{\bar x, d}(x)}
\end{equation*}
holds for all $x\in \bb R^n$ so the right-hand side of equation \eqref{eq:entire_nonlocal} is independent of $\mu$. In particular, Theorem \ref{theorem:classification} implies that $V = |\bb S^n|$. Thus, under assumption \eqref{eq:V_lower_bound_assumption}, the condition $u(x) = \circ(|x|^2)$ as $|x|\to\infty$ is necessary and sufficient for solutions to \eqref{eq:entire_nonlocal} to be spherical (again in this nonlocal context ``spherical'' means the sum of a spherical solution to \eqref{eq:entire_local} and a constant, where the constant is chosen so that $V = |\bb S^n$|). With the exception of the case $n = 2$ where every solution to \eqref{eq:entire_nonlocal} is spherical, 
we do not know whether the same is true when $\frac V{|\bb S^n|}\in (0, \frac{n - \mu}{2n - \mu}]$.

As $\mu\to 0^+$ there is a heuristic consistency between Theorem \ref{theorem:classification} and the classification of solutions to \eqref{eq:entire_local} for which $u(x) = \circ(|x|^2)$ as $|x|\to\infty$. Indeed, letting $\mu\to 0^+$ in \eqref{eq:entire_nonlocal} one obtains the problem
\begin{equation*}
	(-\lap)^{n/2}u = (n - 1)!\|e^{nu}\|_{L^1(\bb R^n)}e^{nu},
\end{equation*}
so the function $v = u + \frac 1n\log\|e^{nu}\|_{L^1(\bb R^n)}$ satisfies both \eqref{eq:entire_local} and $v(x) = \circ(|x|^2)$ with $\|e^{nv}\|_{L^1(\bb R^n)} = \|e^{nu}\|_{L^1(\bb R^n)}^2$. The classification of solutions to \eqref{eq:entire_local} gives $v= U_{\bar x, d}$ for some $(\bar x, d)\in \bb R^n\times(0, \infty)$ and thus $|\bb S^n|= \|e^{nv}\|_{L^1(\bb R^n)} = \|e^{nu}\|_{L^1(\bb R^n)}^2$. In particular, 
\begin{equation*}
	u = U_{\bar x, d} - \frac 1{2n}\log|\bb S^n|,  
\end{equation*}
which is precisely what one obtains by letting $\mu\to 0^+$ in \eqref{eq:nonlocal_bubbles}.

The proof of Theorem \ref{theorem:asymptotic} is based on establishing a one-sided inverse $K$ for $(-\lap)^{n/2}$, estimating the asymptotic behavior of $K(I_\mu[e^{\varrho u}]e^{\varrho u})(x)$ as $|x|\to\infty$, and analyzing the elements of $\ker((-\lap)^{n/2})$ that are of the form $u - K(I_\mu[e^{\varrho u}]e^{\varrho u})$ for some solution $u$ to \eqref{eq:entire_nonlocal}. Many of the technicalities of the proof have already been addressed in \cite{Lin1998, Martinazzi2009, Jin2015, Hyder2019}. Our proof of Theorem \ref{theorem:asymptotic} is a modification of the approach in these references that can accommodate the nonlocal term present in problem \eqref{eq:entire_nonlocal}. 

Theorem \ref{theorem:classification} is established via the method of moving spheres. We note that in \cite{HuangNiu2023} a classification of solutions to problem \eqref{eq:entire_nonlocal} was established, but with the single integrability assumption $e^u\in L^n(\bb R^n)$ replaced by the pair of assumptions $e^{\varrho u}\in L^1(\bb R^n)$ and $I_\mu[e^{\varrho u}]e^{\varrho u}\in L^1(\bb R^n)$. However, our motivation for classifying solutions comes from the application of such a classification in describing asymptotic behavior of blow-up sequences to problems locally modeled by \eqref{eq:entire_nonlocal}. In this context, a classification theorem is only useful if it is established under assumptions that are compatible with the symmetries of the problem. While the integrability assumption $e^{\varrho u}\in L^1(\bb R^n)$ used in \cite{HuangNiu2023} is incompatible with the natural rescaling of problem \eqref{eq:entire_nonlocal} given in \eqref{eq:natural_rescaling}, our assumption $e^u\in L^n(\bb R^n)$ is compatible with this rescaling. There is also a technical difference between our method and the method of \cite{HuangNiu2023}. Because $\varrho < n$, the assumption $e^{\varrho u}\in L^1(\bb R^n)$ used in \cite{HuangNiu2023} is stronger at infinity than our assumption $e^{nu}\in L^1(\bb R^n)$. Moreover, this stronger assumption allows the authors in \cite{HuangNiu2023} to obtain a stronger decay estimate of $I_\mu[e^{\varrho u}](x)$ as $|x|\to\infty$ than we can obtain by assuming $e^{nu}\in L^1(\bb R^n)$. The stronger decay estimate for $I_\mu[e^{\varrho u}]$ allows one to derive a Pohozaev identity from which the exact decay rate for solutions $u$ can be computed. Our weaker decay estimate for $I_\mu[e^{\varrho u}]$ does not permit the use of a Pohozaev identity to capture the precise decay rate of $u$. To circumvent this, we use two applications of the method of moving spheres to capture the exact decay rate; one to ensure that decay is not too fast and the other to ensure that decay is not too slow. 

This manuscript is organized as follows. In Section \ref{s:preliminaries} we discuss some preliminary notions including the definition of $(-\lap)^{n/2}$ and the notion of distributional solution. Section \ref{s:preliminaries} also introduces and gives relevant properties of a one-sided inverse for $(-\lap)^{n/2}$ that will be needed in the sequel. Section \ref{s:regularity_decay} is devoted to the regularity and decay properties of distributional solutions to \eqref{eq:entire_nonlocal}. The proof of Theorem \ref{theorem:asymptotic}, which follows from these regularity and decay properties, is also presented in Section \ref{s:regularity_decay}. Section \ref{s:integral_representations} establishes some integral representations of solutions to \eqref{eq:entire_nonlocal} and their Kelvin transforms under assumptions \eqref{eq:V_lower_bound_assumption} and \eqref{eq:asymtpotic_growth_assumption}. In Section \ref{s:precise_decay} these integral representations are used to compute the precise decay rate of solutions to \eqref{eq:entire_nonlocal} satisfying both \eqref{eq:V_lower_bound_assumption} and \eqref{eq:asymtpotic_growth_assumption}. This decay rate is used in Section \ref{s:classification} where the proof of Theorem \ref{theorem:classification} is presented. 

Throughout the manuscript we assume that $n\in \bb N$, that $\mu\in (0, n)$ and that $\varrho$ is as in \eqref{eq:rho}. These assumptions are to be understood even when they are not explicitly stated in lemmas, propositions, etc. We use $C$ to denote various positive constants whose value may change from one line to the next and even within the same line. When it is important to do so we will emphasize the quantities on which $C$ depends in the notation.

\section{Preliminaries}
\label{s:preliminaries}
%
For $s>0$ we understand $(-\lap)^s$ as the $\mc S'(\bb R^n)$-valued map defined on 
\begin{equation*}
	L_s(\bb R^n)
	= \left\{u\in L^1_{\loc}(\bb R^n): \int_{\bb R^n}\frac{|u(x)|}{1 + |x|^{n + 2s}}\; \d x< \infty\right\}
\end{equation*}
and given by 
\begin{equation*}
	\lb(-\lap)^s u, \varphi\rb
	= \int_{\bb R^n}u (-\lap)^s\varphi
	\qquad \text{ for }\varphi\in \mc S(\bb R^n),  
\end{equation*}
where $\mc S(\bb R^n)$ is the Schwartz space of rapidly decreasing functions and $\mc S'(\bb R^n)$ is its dual. The convergence of the integral appearing on the right-hand side of this equality is guaranteed by Proposition 2.1 of \cite{Hyder2019}.
\begin{defn}
If $f\in L^1(\bb R^n)$, a \emph{distributional solution} to $(-\lap)^{n/2}u = f$ in $\bb R^n$ is a function $u\in L_{n/2}(\bb R^n)$ for which 
\begin{equation}
\label{eq:distributional_solution}
	\int_{\bb R^n}u(-\lap)^{n/2}\varphi
	= \int_{\bb R^n}f\varphi
\end{equation}
for all $\varphi \in \mc S(\bb R^n)$. 
\end{defn}
If $n$ is even the assumption $u\in L_{n/2}(\bb R^n)$ is sufficient but not necessary for the integral on the left-hand side of \eqref{eq:distributional_solution} to be meaningful. Since the assumption $u\in L_{n/2}(\bb R^n)$ will be need for all $n$ to establish Theorems \ref{theorem:asymptotic} and \ref{theorem:classification} (specifically, it is needed for Lemma \ref{lemma:KernelClassification}) we include this assumption in the definition of distributional solution. 

The sharp Hardy-Littlewood-Sobolev inequality will play a key role in the sequel. For brevity we refer to this inequality as the HLS inequality. 
\begin{oldtheorem}[HLS inequality]
Let $\mu\in (0, n)$ and suppose $p, q\in (1, \infty)$ satisfy $\frac 1q = \frac 1p - \frac{n - \mu}n$. There is an optimal constant $\mc H = \mc H(n, \mu, p)>0$ such that the inequality 
\begin{equation*}
	\|I_\mu[f]\|_{L^q(\bb R^n)}
	\leq \mc H(n, \mu, p)\|f\|_{L^p(\bb R^n)}
\end{equation*}
holds for all $f\in L^p(\bb R^n)$. 
\end{oldtheorem}
In the special case that $p = n/\varrho = 2n/(2n - \mu)$ and $q = 2n/\mu$ we use the shortened notation $\mc H(n, \mu) = \mc H(n, \mu, n/\varrho)$. In this case, as shown in \cite{Lieb1983}, the extremal functions are of the form 
\begin{equation}
\label{eq:HLS_extremal}
	f(x) = \pm\left(\frac a{d^2 + |x - \bar x|^2}\right)^{\varrho}
\end{equation}
for some $a\in (0, \infty)$ and some $(\bar x, d)\in \bb R^n\times (0, \infty)$ and the value of the sharp constant is
\begin{equation}
\label{eq:sharp_HLS_constant}
	\mc H(n, \mu) 
	= \pi^{\mu/2}\frac{\Gamma\left(\frac{n + \mu}2\right)}{\Gamma\left(n - \frac \mu 2\right)}\left[\frac{\Gamma\left(\frac n2\right)}{\Gamma(n)}\right]^{-1 + \mu/n}. 
\end{equation}
The next lemma show that the integrability condition in \eqref{eq:entire_nonlocal} guarantees the finiteness of $V$ as defined in \eqref{eq:nonlocal_volume}. The proof is a simple consequence of H\"older's inequality and the HLS inequality and is therefore omitted. 
\begin{lemma}
\label{lemma:RHS_L1_integrable}
There is a constant $C = C(n, \mu)>0$ such that for all $u:\bb R^n\to \bb R$ for which $e^u\in L^n(\bb R^n)$ the estimate
\begin{equation*}
	\|I_\mu[e^{\varrho u}]e^{\varrho u}\|_{L^1(\bb R^n)}
	\leq C\|e^u\|_{L^n(\bb R^n)}^{2\varrho}
\end{equation*}
holds. In particular if $u\in L_{n/2}(\bb R^n)$ with $e^u\in L^n(\bb R^n)$ then problem \eqref{eq:entire_nonlocal} has a meaning in the sense of distributions. 
\end{lemma}
\subsection{A one-sided inverse for $(-\lap)^{n/2}$}
In this subsection we define and state properties of a one-sided inverse for $(-\lap)^{n/2}$. Accordingly for $f\in L^1(\bb R^n)$ we define 
\begin{equation}
\label{eq:KernelOperator}
	Kf(x)
	= c_n\int_{\bb R^n}\log\left(\frac{1 + |y|}{|x - y|}\right)f(y)\; \d y, 
\end{equation}
where 
\begin{equation}
\label{eq:cn}
	c_n = \frac{2}{(n - 1)!|\bb S^n|}
\end{equation} 
is the constant for which $-c_n\log|x|$ is the fundamental solution for $(-\lap)^{n/2}$ in $\bb R^n$ in the sense that 
\begin{equation*}
	(-\lap)^{n/2}(-c_n\log|x|) = \delta_0(x). 
\end{equation*}
As we will see in Lemma \ref{lemma:NonhomogeneousSolution}, $K$ is a one-sided inverse for $(-\lap)^{n/2}$. 
For any multiindex $\alpha$ of length $n$ and for which $\abs\alpha\geq 1$ we use the notation 
\begin{equation}
\label{eq:KernelDerivatives}
	K_\alpha(x)
	= 
	- c_n\partial^\alpha\log\abs x. 
\end{equation}
Evidently $K_\alpha\in C^\infty(\bb R^n\setminus\{0\})$ satisfies $\partial^\beta K_\alpha = K_{\alpha + \beta}$ for any multiindex $\beta$ of length $n$ and the estimate $\abs{K_\alpha(x)}\leq C(\abs \alpha)\abs x^{-\abs \alpha}$. The proofs of Lemmata \ref{lemma:MappingProps}, \ref{lemma:NonhomogeneousSolution} and \ref{lemma:KernelClassification} below can be found in \cite{Gluck2020classification}. 
\begin{lemma}
\label{lemma:MappingProps}
Let $K$ be the operator defined in \eqref{eq:KernelOperator}. 
\begin{enumerate}
	\item If $f\in L^1(\bb R^n)$ then $Kf\in W^{n - 1, 1}_{\rm loc}(\bb R^n)\cap L_{s}(\bb R^n)$ for any $s>0$ and 
	\begin{eqnarray}
	\label{eq:DerivativesInside}
	\partial^\alpha Kf(x)
		& = & 
		\int_{\bb R^n} K_\alpha(x - y)f(y)\; \d y
	\end{eqnarray}
	for any multiindex $\alpha$ satisfying $1\leq \abs \alpha \leq n - 1$, where the equality holds in the sense of $L^1_{\rm loc}(\bb R^n)$. Moreover, for any such $\alpha$ and $s$ we have $\partial^\alpha Kf\in L_s(\bb R^n)$. 
	\item If $f\in L^1\cap L_{\rm loc}^p(\bb R^n)$ for some $p>n$ then $Kf\in W^{n - 1, \infty}_{\rm loc}(\bb R^n)$. 
	\item If $f\in L^1\cap L^\infty(\bb R^n)$ then $Kf\in C^{n - 1}(\bb R^n)$. 
\end{enumerate}
\end{lemma}
\begin{lemma}
\label{lemma:NonhomogeneousSolution}
If $f\in L^1(\bb R^n)$ then $Kf$ as defined in \eqref{eq:KernelOperator} is a distributional solution to 
\begin{equation}
\label{eq:Nonhomogeneous}
	(-\lap)^{n/2} u = f
	\qquad
	\text{ in }\bb R^n. 
\end{equation}
\end{lemma}
\begin{lemma}
\label{lemma:KernelClassification}
If $p$ is a distributional solution to $(-\lap)^{n/2}p =0$ in $\bb R^n$ then $p$ is a polynomial whose degree does not exceed $n - 1$. 
\end{lemma}
The following lemma shows that if $e^u\in L^n(\bb R^n)$ then $I_\mu[e^{\varrho u}]e^{\varrho u}$ inherits the same degree of smoothness possessed by $u$. Since its proof is routine, the details are omitted. 
\begin{lemma}
\label{lemma:Imu_Ck}
For any $k\in \bb N$, if $u\in C^k(\bb R^n)$ and $e^u\in L^n(\bb R^n)$ then $I_\mu[e^{\varrho u}]e^{\varrho u}\in C^k(\bb R^n)$. 
\end{lemma}
\section{Asymptotic Behavior of Solutions at Infinity}
\label{s:regularity_decay}
This section is devoted to the proof of Theorem \ref{theorem:asymptotic}. Along the way we will prove a variety of intermediate results for solutions to \eqref{eq:entire_nonlocal} including smoothness and upper boundedness. We remind the reader that in the lemmata and propositions that follow, we always assume that $\mu\in (0, n)$ and $\varrho$ is as in \eqref{eq:rho}, even when these assumptions are not explicitly stated. 
\begin{prop}
\label{prop:distributional_solutions_smooth}
If $u$ is a distributional solution to \eqref{eq:entire_nonlocal} then $u\in C^\infty(\bb R^n)$. 
\end{prop}
\begin{proof}
The proof will be carried out in two steps. In Step 1 we will show that $u\in C^{n - 2}(\bb R^n)$ and in Step 2 we will improve the smoothness of $u$ to $u\in C^\infty(\bb R^n)$. 
\begin{enumerate}[label = {\bf Step \arabic*.}, ref = {Step \arabic*}, wide = 0pt]
	\item \label{item:u_Cn-2} For ease of notation we set 
	\begin{equation*}
		f = (n - 1)!I_\mu[e^{\varrho u}]e^{\varrho u}. 
	\end{equation*}
	Fix 
	\begin{equation}
	\label{eq:q_largeness}
		q> \max\left\{n \varrho, \frac{n\varrho}{n - \mu}\right\}
	\end{equation}
	and decompose $f$ as $f = f_1 + f_2$ for some $0\leq f_1\in L^1\cap L^\infty(\bb R^n)$ and some $0\leq f_2$ satisfying
	\begin{equation}
	\label{eq:f2_norm_smallness}
		\|f_2\|_1:= \|f_2\|_{L^1(\bb R^n)}< \frac{n}{qc_n}, 
	\end{equation}
	where $c_n$ is as in \eqref{eq:cn}. For $i = 1, 2$ define $v_i = Kf_i$. Lemma \ref{lemma:MappingProps} guarantees that $v_1\in C^{n - 1}(\bb R^n)$ and $v_2\in W^{n - 1, 1}_{\loc}\cap L_{n/2}(\bb R^n)$. Lemma \ref{lemma:NonhomogeneousSolution} guarantees that $p:= u - (v_1 + v_2)$ satisfies $(-\lap)^{n/2}p = 0$ in the distributional sense on $\bb R^n$, so Lemma \ref{lemma:KernelClassification} guarantees that $p$ is a polynomial. In particular $p + v_1\in C^{n - 1}(\bb R^n)$ so to show that $u\in C^{n - 2}(\bb R^n)$ it suffices to show that $v_2\in C^{n - 2}(\bb R^n)$. To do so we first claim that $e^{v_2}\in L^q_{\loc}(\bb R^n)$ and that there is a constant $C(n)>0$ such that the estimate
	\begin{equation}
	\label{eq:ev2_Lq_loc}
		\|e^{v_2}\|_{L^q(B_R)}\leq C(n)R^{n/q}
	\end{equation}
	holds for all $R\gg 1$. To verify this claim fix $R\gg 1$ and observe first that an application of Jensen's inequality gives
	\begin{equation}
	\label{eq:eqv2_split}
	\begin{split}
		0
		& \leq \int_{B_R}e^{qv_2}\\
		& \leq \int_{B_R}\left(\int_{B_{2R}} + \int_{\bb R^n\setminus B_{2R}}\right)\left(\frac{1 + |y|}{|x - y|}\right)^{qc_n\|f_2\|_1}\frac{f_2(y)}{\|f_2\|_1}\; \d y\; \d x. 
	\end{split}
	\end{equation}
	Since $R\gg 1$ and by the smallness condition on $\|f_2\|_1$ in \eqref{eq:f2_norm_smallness}, we obtain 
	\begin{equation}
	\label{eq:eqv2_split_term1}
	\begin{split}
		\int_{B_R}\int_{B_{2R}}& \left(\frac{1 + |y|}{|x - y|}\right)^{qc_n\|f_2\|_1}\frac{f_2(y)}{\|f_2\|_1}\; \d y\; \d x\\
		& 
		\leq (3R)^{qc_n\|f_2\|_1}\int_{B_{2R}}\frac{f_2(y)}{\|f_2\|_1}\int_{B_{4R}(y)}|x - y|^{-qc_n\|f_2\|_1}\; \d x\; \d y\\
		& \leq C(n) R^n. 
	\end{split}
	\end{equation}
	Moreover, since $\frac{1 + |y|}{|x- y|}\leq 4$ whenever $x\in B_R$ and $y\in \bb R^n\setminus B_{2R}$ we have
	\begin{equation}
	\label{eq:eqv2_split_term2}
	\begin{split}
		\int_{B_R}\int_{\bb R^n\setminus B_{2R}}& \left(\frac{1 + |y|}{|x - y|}\right)^{qc_n\|f_2\|_1}\frac{f_2(y)}{\|f_2\|_1}\; \d y\; \d x\\
		& \leq 4^{qc_n\|f_2\|_1}|B_R|\\
		& \leq 4^n|B_R|\\
		& \leq C(n) R^n. 
	\end{split}
	\end{equation}
	Bringing estimates \eqref{eq:eqv2_split_term1} and \eqref{eq:eqv2_split_term2} back to \eqref{eq:eqv2_split} establishes \eqref{eq:ev2_Lq_loc}. 
	Next we claim that 
	\begin{equation}
	\label{eq:nonlocality_locally_bounded}
		I_\mu[e^{\varrho u}]\in L^\infty_{\loc}(\bb R^n). 
	\end{equation}
	To see this, fix $R\gg 1$ and let $x\in B_R$. Setting $A = \|e^{\varrho(v_1 + p)}\|_{L^\infty(B_{2R})}$ and with two applications of H\"older's inequality we have 
	\begin{equation*}
	\begin{split}
		I_\mu[e^{\varrho u}](x)
		\leq & \; A\int_{B_{2R}}\frac{e^{\varrho v_2(y)}}{|x- y|^\mu}\; \d y + 2^\mu\int_{\bb R^n\setminus B_{2R}}\frac{e^{\varrho u(y)}}{|y|^\mu}\; \d y\\
		\leq & \; A\|e^{v_2}\|_{L^q(B_{2R})}^\varrho\left(\int_{B_{4R}(x)}|x - y|^{-\frac{\mu q}{q - \varrho}}\; \d y\right)^{1 - \varrho/q}\\
		& + 2^\mu\|e^u\|_{L^n(\bb R^n)}^\varrho\left(\int_{\bb R^n\setminus B_R}|y|^{-2n}\; \d y\right)^{1 - \varrho/n}\\
		\leq & \; C\left(R^{n - \mu} + R^{-\mu/2}\right),  
	\end{split}
	\end{equation*}
	where estimate \eqref{eq:ev2_Lq_loc} was used in the final inequality and $C = C(n, \mu, A, \|e^u\|_{L^n(\bb R^n)})$. Since $R\gg 1$ is arbitrary, \eqref{eq:nonlocality_locally_bounded} is established. Now for any $R\gg 1$ using \eqref{eq:nonlocality_locally_bounded} and \eqref{eq:ev2_Lq_loc} we have
	\begin{equation*}
	\begin{split}
		\int_{B_R}\left|I_\mu[e^{\varrho u}]e^{\varrho u}\right|^{q/\varrho}
		& \leq \|I_\mu[e^{\varrho u}]\|_{L^\infty(B_R)}^{q/\varrho}\int_{B_R}e^{qu}\\
		& \leq A^{q/\varrho}\|I_\mu[e^{\varrho u}]\|_{L^\infty(B_R)}^{q/\varrho}\int_{B_R}e^{qv_2}\\
		& < \infty. 
	\end{split}
	\end{equation*}
	Since $q>n\varrho$ Lemma \ref{lemma:MappingProps} guarantees that $v_1 + v_2 = Kf\in W^{n -1, \infty}_{\loc}(\bb R^n)\subset C^{n - 2}(\bb R^n)$. Combining this containment with the containment $v_1\in C^{n - 1}(\bb R^n)$ shows that $v_2\in C^{n - 2}(\bb R^n)$ thereby completing the proof of \ref{item:u_Cn-2}. 
	\item To show that $u\in C^{\infty}(\bb R^n)$ it is sufficient to show that $v_1 + v_2= Kf$ is smooth. By Lemma \ref{lemma:MappingProps}, for any $1\leq \abs{\alpha}\leq n-1$ 
	\begin{equation*}
		\partial^\alpha(v_1 + v_2)(x)
		= \int_{\bb R^n}K_\alpha(x - y)f(y)\; \d y,
	\end{equation*}
	where $K_\alpha$ is as in \eqref{eq:KernelDerivatives}. Let $\eta\in C^\infty(\bb R^n)$ satisfy $0\leq \eta\leq 1$, $\eta (x) = 0$ for $\abs x\leq 1$, and $\eta(x) = 1$ for $\abs x\geq 2$. Then $\eta K_\alpha\in C^\infty\cap L^\infty(\bb R^n)$ and 
	\begin{equation*}
	\begin{split}
		\partial^\alpha& (v_1 + v_2)(x)\\
		&= \int_{\bb R^n} \eta(x - y)K_\alpha(x - y)f(y)\; \d y +\int_{\bb R^n}(1 - \eta(y)) K_\alpha(y)f(x-y)\; \d y.
	\end{split}
	\end{equation*}
	Given $u\in C^k(\bb R^n)$ for some $k\in \bb N$, if $\beta$ is a multiindex for which $\abs\beta = k$ then 
	\begin{equation*}
	\begin{split}
		\partial^{\alpha + \beta}(v_1 + v_2)(x)
		= & \;  \int_{\bb R^n}\partial_x^\beta(\eta(x - y)K_\alpha(x - y))f(y)\; \d y\\
		&+ \int_{\bb R^n}(1 - \eta(y))K_\alpha(y) \partial_ x^\beta(f(x- y))\; \d y.\\
	\end{split}
	\end{equation*}
	The continuity of each of the two terms on the right-hand side is routinely verified (one may use Lemma \ref{lemma:Imu_Ck} to verify the continuity of the second term) and thus we have $v_1 + v_2\in C^{n - 1+ k}(\bb R^n)$. Starting with $k = n - 2$ we iterate this argument to get $v_1 + v_2\in C^\infty(\bb R^n)$. 
\end{enumerate}
\end{proof}
The proof of the following lemma is similar to that of Lemma 2.1 of \cite{Lin1998} and is therefore omitted. 
\begin{lemma}
\label{lemma:Kf_lower_decay}
There is a universal constant $C>0$ such that for all $0\leq f \in L^1(\bb R^n)$ the estimate
\begin{equation*}
	(n - 1)!Kf(x)
	\geq -\frac{2\|f\|_{L^1(\bb R^n)}}{|\bb S^n|}\left(\log|x| + C\right)
\end{equation*}
holds for all $x\in \bb R^n\setminus B_4$. 
\end{lemma}
\begin{lemma}
\label{lemma:polynomial_part_bounded_above}
If $u$ is a distributional solution to \eqref{eq:entire_nonlocal} then there is a polynomial $p$ that is bounded above and for which $\deg p\leq n - 1$ such that
\begin{equation}
\label{eq:particular+harmonic}
	u = (n - 1)!K\left(I_\mu[e^{\varrho u}]e^{\varrho u}\right) + p, 
\end{equation}
where $K$ is as in \eqref{eq:KernelOperator}. 
\end{lemma}
\begin{proof}
Defining $p$ by 
\begin{equation*}
	p = u - (n - 1)!K\left(I_\mu[e^{\varrho u}]e^{\varrho u}\right), 
\end{equation*}
Lemma \ref{lemma:NonhomogeneousSolution} guarantees that $p$ satisfies $(-\lap)^{n/2} p= 0$ in $\bb R^n$ in the distributional sense so Lemma \ref{lemma:KernelClassification} implies that $p$ is a polynomial for which $\deg p\leq n -1$. It remains to show that $p$ is bounded above. Proceeding by way of contradiction, suppose $\sup_{\bb R^n}p = +\infty$. Theorem 3.1 of \cite{Gorin1961} guarantees the existence of $s>0$ for which $\lim_{r\to\infty}r^{-s}\sup_{\bdy B_r}p = +\infty$. For $r$ large using Lemma \ref{lemma:Kf_lower_decay} we have
\begin{equation*}
\begin{split}
	\sup_{\bdy B_r}p
	& 
	\leq \sup_{\bdy B_r}u + \frac{2V}{|\bb S^n|}\log r + C,  
\end{split}
\end{equation*}
where $V$ is as in \eqref{eq:nonlocal_volume}. Therefore, 
\begin{equation*}
\begin{split}
	+\infty
	& = \liminf_{r\to\infty}r^{-s}\sup_{\bdy B_r}p\\
	& \leq \liminf_{r\to\infty}r^{-s}\left(\sup_{\bdy B_r}u + \frac{2V}{|\bb S^n|}\log r + C\right)\\
	& = \liminf_{r\to\infty}r^{-s}\sup_{\bdy B_r}u. 
\end{split}
\end{equation*}
For $r$ large, let $x_r\in \bdy B_r$ satisfy $p(x_r) = \sup_{\bdy B_r}p$.
Since $\deg p\leq n- 1$ there is $C>1$ such that $|\Grad p(x)|\leq C(1 + |x|^{n - 2})$ for all $x\in \bb R^n$. Therefore, for any $y\in B_{r^{2-n}}(x_r)$ we have
\begin{equation*}
\begin{split}
	|p(x_r) -p(y)|
	& 
	\leq \|\Grad p\|_{L^\infty(B_{r^{2-n}}(x_r))}|x_r - y|\\
	& \leq \|\Grad p\|_{L^\infty(B_{2r}\setminus B_{r/2})}r^{2-n}\\
	& 
	\leq C 
\end{split}
\end{equation*}
for some positive constant $C$ that is independent of $r$. 
Using this estimate together with Lemma \ref{lemma:Kf_lower_decay} we find that for every $y\in B_{r^{2- n}}(x_r)$ there holds
\begin{equation*}
\begin{split}
	r^{-s}u(y)
	& 
	\geq r^{-s}p(x_r) - \frac{2V}{|\bb S^n|}r^{-s}\log r - Cr^{-s}\\
	& = r^{-s}p(x_r)+ \circ(1)
\end{split}
\end{equation*}
as $r\to\infty$. Since $r^{-s}p(x_r)\to\infty$, this estimate guarantees the existence of $R\gg 1$ such that the inequality $u(y)\geq r^s$ holds whenever $r\geq R$ and $y\in B_{r^{2-n}}(x_r)$. Now for $r$ large, 
\begin{equation}
\label{eq:V_lower}
	V
	\geq \int_{B_{\frac{r^{2 - n}}4}(x_r)}I_\mu[e^{\varrho u}]e^{\varrho u}
	\geq e^{\varrho r^s}\int_{B_{\frac{r^{2 - n}}4}(x_r)}I_\mu[e^{\varrho u}]. 
\end{equation}
Moreover, for any $x\in B_{r^{2-n}/4}(x_r)$ we have
\begin{equation*}
\begin{split}
	I_\mu[e^{\varrho u}](x)
	& \geq \int_{B_{\frac{r^{2 - n}}2}(x)}\frac{e^{\varrho u(y)}}{|x- y|^\mu}\; \d y\\
	& \geq e^{\varrho r^s}\int_{B_{\frac{r^{2 - n}}2}(x)}|x- y|^{-\mu}\; \d y\\
	& = \frac{|\bb S^n|}{2^{n - \mu}(n - \mu)}r^{-(n - 2)(n-\mu)}e^{\varrho r^s}. 
\end{split}
\end{equation*}
Bringing this back to \eqref{eq:V_lower} we find that there is an $r$-independent constant $C>0$ such that for all $r$ sufficiently large, 
\begin{equation*}
	V
	\geq C r^{-(n - 2)(n - \mu)}e^{2\varrho r^s}
\end{equation*}
Choosing $r$ sufficiently large 
 we obtain $V\geq V+1$ which is a contradiction. 
\end{proof}
\begin{lemma}
\label{lemma:Kf_upper_decay_estimate}
If $u$ is a distributional solution to \eqref{eq:entire_nonlocal} then for every $\epsilon>0$ there is $R>0$ such that the estimate 
\begin{equation}
\label{eq:Kf_upper_decay_estimate}
	(n - 1)!K\left(I_\mu[e^{\varrho u}]e^{\varrho u}\right)(x)
	\leq -\left(\frac{2V}{|\bb S^n|} - \epsilon\right)\log|x|
\end{equation}
holds for all $x\in \bb R^n\setminus B_R$. 
\end{lemma}
Before proving Lemma \ref{lemma:Kf_upper_decay_estimate} we state and prove a lemma that will be used in the proof of Lemma \ref{lemma:Kf_upper_decay_estimate}. 
\begin{lemma}
\label{lemma:for_Kf_upper_decay}
Suppose $u$ is a distributional solution to \eqref{eq:entire_nonlocal}. For every $q\geq 1$ and for every small $\epsilon_1, \epsilon_2>0$ there is a positive constant $C = C(q, \epsilon_1, \epsilon_2)$ and there is $R = R(q, \epsilon_1, \epsilon_2)\gg 1$ such that the estimate 
\begin{equation*}
	\frac{1}{s^{n - \epsilon_2}}\int_{B_s(x)}\exp\left(qK(I_\mu[e^{\varrho u}]e^{\varrho u})\right)
	\leq \frac{C}{|x|^{c_nq(V - \epsilon_1)}}
\end{equation*}
holds for all $s\in (0, 1]$ and all $x\in \bb R^n\setminus B_R$, where $c_n$ is as in \eqref{eq:cn}.
\end{lemma}
\begin{proof}
For ease of notation we set
\begin{equation}
\label{eq:temp_v}
	v= K\left(I_\mu[e^{\varrho u}]e^{\varrho u}\right). 
\end{equation}
Arguing as in Lemma 2.4 of \cite{Lin1998} we find that for every $\epsilon>0$ there is $R\gg 1$ such that 
\begin{equation}
\label{eq:as_in_Lin}
	c_n^{-1}v(x) + (V - \epsilon)\log|x|
	\leq \int_{B_1(x)}\log\frac{1}{|x - y|} I_\mu[e^{\varrho u}](y)e^{\varrho u(y)}\; \d y
\end{equation}
whenever $x\in \bb R^n\setminus B_R$. Fix $\epsilon_1\in (0, V)$ and choose $R\gg 1$ such that \eqref{eq:as_in_Lin} holds with $\epsilon = \epsilon_1$. Fix $q\geq 1$ and $\epsilon_2\in (0, n)$. After increasing $R$ if necessary we may assume in addition that 
\begin{equation*}
	r:= 
	c_nq\|I_\mu[e^{\varrho u}]e^{\varrho u}\|_{L^1(\bb R^n\setminus B_R)}
	< \epsilon_2. 
\end{equation*}
Define 
\begin{equation*}
	\d \eta(y)
	= \frac{I_\mu[e^{\varrho u}](y)e^{\varrho u(y)}}{\|I_\mu[e^{\varrho u}]e^{\varrho u}\|_{L^1(\bb R^n\setminus B_R)}}\; \d y
\end{equation*}
and apply both \eqref{eq:as_in_Lin} with $\epsilon= \epsilon_1$ and Jensen's inequality to find that if $|x|> R + 1$ then 
\begin{equation*}
\begin{split}
	|x|^{c_nq(V - \epsilon_1)}e^{qv(x)}
	& \leq \exp\left(\int_{\bb R^n\setminus B_R}\log(|x - y|^{-r})\chi_{B_1(x)}(y)\; \d \eta(y)\right)\\
	& \leq \int_{\bb R^n\setminus B_R}\exp\left(\log(|x - y|^{-r})\chi_{B_1(x)}(y)\right)\; \d \eta(y)\\
	& 
	\leq 1 + \int_{B_1(x)}|x- y|^{-r}\; \d \eta(y). 
\end{split}
\end{equation*}
Using this estimate we find that for any $z\in \bb R^n\setminus B_{R + 2}$ and any $s\in (0, 1]$ there holds
\begin{equation}
\label{eq:zpower_integral}
\begin{split}
	|z|^{c_nq(V - \epsilon_1)}& \int_{B_s(z)}e^{qv(x)}\; \d x\\
	& \leq |z|^{c_nq(V - \epsilon_1)}\int_{B_s(z)}|x|^{-c_nq(V - \epsilon_1)}\left(1 + \int_{B_1(x)}|x - y|^{-r}\; \d \eta(y)\right)\; \d x\\
	& \leq C\int_{B_s(z)}\left(1 + \int_{B_1(x)}|x - y|^{-r}\; \d \eta(y)\right)\; \d x\\
	& \leq C\left(s^n + \int_{B_s(z)}\int_{B_1(x)}|x - y|^{-r}\; \d \eta(y)\; \d x\right)\\
	& = C\left(s^n + \int_{\bb R^n\setminus B_R}\int_{B_s(z)}|x- y|^{-r}\chi_{B_1(x)}(y)\; \d x\; \d \eta(y)\right). 
\end{split}
\end{equation}
For every $y\in \bb R^n\setminus B_R$, every $z\in \bb R^n\setminus B_{R + 2}$ and every $s\in (0, 1]$ we have
\begin{equation}
\label{eq:Brhoz_integral}
	\int_{B_s(z)}|x- y|^{-r}\chi_{B_1(x)}(y)\; \d x
	= \int_{B_s(z)\cap B_1(y)}|x- y|^{-r}\; \d x. 
\end{equation}
Setting $D(y, z) = \{x\in \bb R^n: |x- y|\geq 2|x - z|\}$ and in view of the assumption $r< n$ we have 
\begin{equation*}
	\int_{B_s(z)\cap B_1(y)\cap D(y, z)}|x- y|^{-r}\; \d x
	\leq \int_{B_s(z)}|x- z|^{-r}\; \d x
	\leq C(n)s^{n - r}. 
\end{equation*}
Moreover, $B_s(z)\setminus D(y, z)\subset B_{2s}(y)$ so 
\begin{equation*}
	\int_{(B_s(z)\cap B_1(y))\setminus D(y, z)}|x- y|^{-r}\; \d x
	\leq \int_{B_{2s}(y)}|x- y|^{-r}\; \d x
	\leq C(n)s^{n - r}. 
\end{equation*}
Bringing the previous two estimates back to \eqref{eq:Brhoz_integral} we find that
\begin{equation*}
	\int_{B_s(z)}|x- y|^{-r}\chi_{B_1(x)}(y)\; \d x
	\leq Cs^{n - r}
\end{equation*}
whenever $y\in \bb R^n\setminus B_R$, $z\in \bb R^n\setminus B_{R + 2}$ and $s\in (0, 1]$. Now returning to \eqref{eq:zpower_integral} we find that for all $z\in \bb R^n\setminus B_{R + 2}$ and all $s\in (0, 1]$, 
\begin{equation*}
\begin{split}
	|z|^{c_nq(V - \epsilon_1)}\int_{B_s(z)}e^{qv(x)}\; \d x
	& \leq C\left(s^n + \int_{\bb R^n\setminus B_R}s^{n - r}\; \d \eta(y)\right)\\
	& \leq Cs^{n - \epsilon_2}, 
\end{split}
\end{equation*}
where, in the final estimate we used $s\in [0, 1)$ and and $\epsilon_2< r< n$. The asserted estimate follows. 
\end{proof}
With Lemma \ref{lemma:for_Kf_upper_decay} in hand we are ready to prove Lemma \ref{lemma:Kf_upper_decay_estimate}. 
\begin{proof}[Proof of Lemma \ref{lemma:Kf_upper_decay_estimate}]
Let $\epsilon>0$. Arguing as in Lemma 2.4 of \cite{Lin1998} we find that there is $R>0$ such that \eqref{eq:as_in_Lin} holds for all $x\in \bb R^n\setminus B_R$. Setting $f = (n - 1)!I_\mu[e^{\varrho u}]e^{\varrho u}$, setting $p = u - Kf$, and with $v$ as in \eqref{eq:temp_v}, for any $x\in \bb R^n$ Lemma \ref{lemma:polynomial_part_bounded_above} and H\"older's inequality give
\begin{equation}
\label{eq:fundamal_solution_centered_unit_ball_estimate}
\begin{split}
	\int_{B_1(x)}& \log\frac{1}{|x- y|}I_\mu[e^{\varrho u}](y)e^{\varrho u(y)}\; \d y\\
	\leq & \; \|e^{\varrho p}\|_{L^\infty(\bb R^n)}\int_{B_1(x)} \log\frac{1}{|x- y|}I_\mu[e^{\varrho u}](y)e^{\varrho Kf(y)}\; \d y\\
	\leq & \; \|e^{\varrho p}\|_{L^\infty(\bb R^n)}\|\log|x - \cdot|\|_{L^r(B_1(x))}\|I_\mu[e^{\varrho u}]\|_{L^{2n/\mu}(\bb R^n)}\left(\int_{B_1(x)}e^{\varrho tKf(y)}\; \d y\right)^{1/t}\\
	\leq & 
	\; C\left(\int_{B_1(x)}e^{\varrho t (n - 1)!v}\right)^{1/t}
\end{split}
\end{equation}
whenever $r, t\in (1, \infty)$ satisfy $\frac 1r + \frac 1t = 1 - \frac\mu{2n}$. In particular, if $|x|> R + 1$ then Lemma \ref{lemma:for_Kf_upper_decay} (applied with $s = 1$) guarantees that 
\begin{equation*}
	\int_{B_1(x)}e^{\varrho t(n - 1)!v}
	\leq \frac{C}{|x|^{c_n\varrho t(n - 1)!(V - \epsilon_1)}},
\end{equation*}
so estimate \eqref{eq:fundamal_solution_centered_unit_ball_estimate} gives
\begin{equation*}
\begin{split}
	\int_{B_1(x)}& \log\frac{1}{|x- y|}I_\mu[e^{\varrho u}](y)e^{\varrho u(y)}\; \d y\\
	\leq & \; \frac C{|x|^{c_n\varrho t(n -1)!(V - \epsilon_1)}}\\
	< & \; \epsilon\log|x|
\end{split} 
\end{equation*}
whenever $|x|$ is sufficiently large. Using this estimate in \eqref{eq:as_in_Lin} gives the asserted estimate. 
\end{proof}
\begin{coro}
\label{coro:bounded_above}
If $u$ is a distributional solution to \eqref{eq:entire_nonlocal} then $u$ is bounded above. 
\end{coro}
\begin{proof}
Setting $f = (n - 1)!I_\mu[e^{\varrho u}]e^{\varrho u}$, Lemma \ref{lemma:polynomial_part_bounded_above} guarantees that $u = Kf + p$, where $p$ is bounded above. Lemma \ref{lemma:Kf_upper_decay_estimate} guarantees the existence of $R\gg 1$ such that $Kf(x)\leq 1$ for all $x\in \bb R^n\setminus B_R$. For any such $R$, Proposition \ref{prop:distributional_solutions_smooth} guarantees that $u\in L^\infty(B_{4R})$. For $x\in B_{2R}$ consider the inequality
\begin{equation}
\label{eq:Kf_IJ_split}
	\frac{|Kf(x)|}{c_n(n - 1)!}
	\leq I(x) + J(x), 
\end{equation}
where 
\begin{equation*}
\begin{split}
	I(x) & = \int_{\bb R^n\setminus B_{4R}}\abs{\log\left(\frac{1 + |y|}{|x- y|}\right)}I_\mu[e^{\varrho u}](y)e^{\varrho u(y)}\; \d y\\
	J(x) & = \int_{B_{4R}}\abs{\log\left(\frac{1 + |y|}{|x- y|}\right)}I_\mu[e^{\varrho u}](y)e^{\varrho u(y)}\; \d y. 
\end{split}
\end{equation*}
Since 
\begin{equation*}
	\frac 14\leq \log\frac{1 + |y|}{|x- y|}\leq 4
\end{equation*}
whenever $x\in B_{2R}$ and $y\in \bb R^n\setminus B_{4R}$ we find that 
\begin{equation*}
	|I(x)|
	\leq 4\int_{\bb R^n\setminus B_{4R}}I_\mu[e^{\varrho u}](y)e^{\varrho u(y)}\; \d y
	\leq C\|I_\mu[e^{\varrho u}]e^{\varrho u}\|_{L^1(\bb R^n)}. 
\end{equation*}
To show $J\in L^\infty(B_{2R})$ we observe that Proposition \ref{prop:distributional_solutions_smooth} implies $u$ is bounded above on $B_{4R}$ and we estimate as follows: 
\begin{equation*}
\begin{split}
	J(x)
	\leq & \; \int_{B_{4R}}\left(\log(5R) + |\log|x- y||\right)I_\mu[e^{\varrho u}](y)e^{\varrho u(y)}\; \d y\\
	\leq & \; \log(5R)\|I_\mu[e^{\varrho u}]e^{\varrho u}\|_{L^1(\bb R^n)}\\
	&  + \|e^{\varrho u}\|_{L^\infty(B_{4R})}\|I_\mu[e^{\varrho u}]\|_{L^{2n/\mu}(\bb R^n)}\left(\int_{B_{8R}(x)}|\log|x- y||^{\frac{2n}{2n - \mu}}\; \d y\right)^{1 - \mu/(2n)}\\
	< & \;\infty.  
\end{split}
\end{equation*}
Bringing the estimates for $I(x)$ and $J(x)$ back to \eqref{eq:Kf_IJ_split} we find that $Kf$ is bounded on $B_{2R}$ and thereby conclude the proof.
\end{proof}
\begin{proof}[Proof of Theorem \ref{theorem:asymptotic}]
Lemma \ref{lemma:Kf_lower_decay} (applied with $f = I_\mu[e^{\varrho u}]e^{\varrho u}$) and Lemma \ref{lemma:Kf_upper_decay_estimate} guarantee that 
\begin{equation*}
	(n -1)!K\left(I_\mu[e^{\varrho u}]e^{\varrho u}\right)(x)
	= -\frac{2V}{|\bb S^n|}\log|x| + \circ(\log|x|)
\end{equation*}
as $|x|\to \infty$. Combining this equality with Lemma \ref{lemma:polynomial_part_bounded_above} completes the proof.
\end{proof}
\section{Integral Representations of Solutions}
\label{s:integral_representations}
The following lemma gives a preliminary decay estimate for $I_\mu[e^{\varrho u}]$ whenever $u$ satisfies both \eqref{eq:entire_nonlocal} and \eqref{eq:V_lower_bound_assumption}. 
\begin{lemma}
\label{lemma:Imu_initial_decay}
If $u$ is a distributional solution to \eqref{eq:entire_nonlocal} for which \eqref{eq:V_lower_bound_assumption} holds then for every $\epsilon>0$ satisfying
\begin{equation}
\label{eq:Imu_decay_epsilon_choice}
	\frac \epsilon 2< 
	\begin{cases}
	\frac V{|\bb S^n|} - \frac{n - \mu}{2n - \mu} & \text{ if } \frac{n - \mu}{2n - \mu} < \frac V{|\bb S^n|}\leq \frac{n}{2n - \mu}\\
	\frac V{|\bb S^n|} - \frac{n}{2n - \mu} & \text{ if } \frac{n}{2n - \mu} < \frac V{|\bb S^n|}\\
	\end{cases}
\end{equation}
there is $C = C(\epsilon)>0$ and there is $R = R(\epsilon)\gg 1$ such that 
\begin{equation}
\label{eq:Imu_initial_decay}
	I_\mu[e^{\varrho u}](x)
	\leq C\left(|x|^{-\mu} + |x|^{n - \mu - \varrho\left(\frac{2V}{|\bb S^n|} - \epsilon\right)}\right)
\end{equation}
whenever $x\in \bb R^n\setminus B_R$. In particular, $I_\mu[e^{\varrho u}]\in L^\infty(\bb R^n)$. 
\end{lemma}
\begin{proof}
Let $u$, $V$ and $\epsilon$ satisfy the hypotheses of the lemma. Choose $R = R(\epsilon)\gg 1$ such that inequality \eqref{eq:Kf_upper_decay_estimate} holds for all $x\in \bb R^n\setminus B_{R/2}$. Thus, for any such $x$ Lemma \ref{lemma:polynomial_part_bounded_above} gives 
\begin{equation}
\label{eq:eu_decay}
\begin{split}
	e^{u(x)}
	& = \exp\left((n - 1)!K(I_\mu[e^{\varrho u}]e^{\varrho u})(x) + p(x)\right)\\
	& \leq \exp\left(\sup_{\bb R^n}p\right)|x|^{-\left(\frac{2V}{|\bb S^n|} - \epsilon\right)}\\
	& \leq C|x|^{-\left(\frac{2V}{|\bb S^n|} - \epsilon\right)}. 
\end{split}
\end{equation}
Fix $y\in \bb R^n\setminus B_R$ and partition $\bb R^n$ as $\bb R^n= \bigcup_{j = 1}^4 D_j(y)$, where
\begin{equation*}
\begin{split}
	D_1 & = B_{R/2}\\
	D_2 & = B_{|y|/2}(y)\\
	D_3 & = B_{2|y|}\setminus(D_1\cup D_2)\\
	D_4 & = \bb R^n\setminus B_{2|y|}. 
\end{split}
\end{equation*}
For convenience we do not indicate the $y$-dependence in the notation for the sets $D_j$. For $j= 1, \ldots, 4$ set $J_j = I_\mu[e^{\varrho u}\chi_{D_j}]$ so that 
\begin{equation}
\label{eq:Imu_decay_decompose}
	I_\mu[e^{\varrho u}] = \sum_{j = 1}^4 J_j. 
\end{equation}
We proceed to separately estimate each $J_j(y)$. If $z\in D_1$ then $|y - z|> |y|/2$ so H\"older's inequality gives
\begin{equation*}
\begin{split}
	J_1(y)
	& \leq \left(\frac 2{|y|}\right)^\mu\int_{B_R} e^{\varrho u(z)}\; \d z\\
	& \leq \left(\frac 2{|y|}\right)^\mu|B_R|^{1 - \varrho/n}\|e^u\|_{L^n(\bb R^n)}^{\varrho}\\
	& \leq C|y|^{-\mu}. 
\end{split}
\end{equation*}
If $z\in D_2$ then $2|z|\geq|y|\geq R$ so using \eqref{eq:eu_decay} we have
\begin{equation*}
\begin{split}
	J_2(y)
	& \leq C\int_{D_2}\frac{|z|^{-\varrho\left(\frac{2V}{|\bb S^n|} - \epsilon\right)}}{|y - z|^\mu}\; \d z\\
	& \leq C|y|^{-\varrho\left(\frac{2V}{|\bb S^n|} - \epsilon\right)}\int_{D_2}|y - z|^{-\mu}\; \d z\\
	& \leq C|y|^{n-\mu - \varrho\left(\frac{2V}{|\bb S^n|} - \epsilon\right)}. 
\end{split}
\end{equation*}
If $z\in D_3$ then $2|y- z|> |y|$ so using \eqref{eq:eu_decay} we have
\begin{equation*}
\begin{split}
	J_3(y)
	& \leq C\left(\frac 2{|y|}\right)^\mu\int_{B_{2|y|}\setminus B_{R/2}}|z|^{-\varrho\left(\frac{2V}{|\bb S^n|} - \epsilon\right)}\; \d z\\
	& \leq C\begin{cases}
		|y|^{n - \mu - \varrho\left(\frac{2V}{|\bb S^n|} - \epsilon\right)} & \text{ if }\frac{n -\mu}{2n - \mu}< \frac{V}{|\bb S^n|}\leq \frac{n}{2n - \mu}\\
		|y|^{-\mu} & \text{ if }\frac{V}{|\bb S^n|}>  \frac{n}{2n - \mu}.
	\end{cases}
\end{split}
\end{equation*}
If $z\in D_4$ then $2|y - z|\geq |z|$ so using \eqref{eq:eu_decay} we have
\begin{equation*}
\begin{split}
	J_4(y)
	& \leq C\int_{D_4}\frac{|z|^{-\varrho\left(\frac{2V}{|\bb S^n|} - \epsilon\right)}}{|y - z|^\mu}\; \d z\\
	& \leq C\int_{\bb R^n\setminus B_{|y|}}|z|^{-\mu -\varrho\left(\frac{2V}{|\bb S^n|} - \epsilon\right)}\; \d z\\
	& \leq C|y|^{n-\mu -\varrho\left(\frac{2V}{|\bb S^n|} - \epsilon\right)}.
\end{split}
\end{equation*}
Bringing the estimates of $J_1, \ldots, J_4$ back to \eqref{eq:Imu_decay_decompose} establishes estimate \eqref{eq:Imu_initial_decay}. To show that $I_\mu[e^{\varrho u}]\in L^\infty(\bb R^n)$, fix $\epsilon>0$ small such that \eqref{eq:Imu_decay_epsilon_choice} holds and choose $R$ and $C$ as in estimate \eqref{eq:Imu_initial_decay} so that $0\leq I_\mu[e^{\varrho u}](y)\leq C$ whenever $y\in \bb R^n\setminus B_R$. For $y\in B_R$, using Corollary \ref{coro:bounded_above} we have 
\begin{equation*}
\begin{split}
	I_\mu[e^{\varrho u}](y)
	= &\; \int_{B_{2R}}\frac{e^{\varrho u(z)}}{|y - z|^\mu}\; \d z + \int_{\bb R^n\setminus B_{2R}}\frac{e^{\varrho u(z)}}{|y - z|^\mu}\; \d z\\
	\leq &\;  \int_{B_{3R}(y)}\frac{e^{\varrho u(z)}}{|y - z|^\mu}\; \d z + \int_{\bb R^n\setminus B_{R}(y)}\frac{e^{\varrho u(z)}}{|y - z|^\mu}\; \d z\\
	\leq &\; \|e^{\varrho u}\|_{L^\infty(\bb R^n)}\int_{B_{3R}(y)}|y - z|^{-\mu}\; \d z\\
	& + \|e^u\|_{L^n(\bb R^n)}^\varrho\left(\int_{\bb R^n\setminus B_R(y)}|y - z|^{-2n}\; \d z\right)^{1 - \varrho/n}\\
	\leq &\; C\|e^u\|_{L^\infty(\bb R^n)}^\varrho R^{n - \mu} + \|e^u\|_{L^n(\bb R^n)}^\varrho R^{-\mu/2}. 
\end{split}
\end{equation*}
\end{proof}
\begin{lemma}
\label{lemma:remove_log_1_y}
If $u$ is a distributional solution to \eqref{eq:entire_nonlocal} satisfying \eqref{eq:V_lower_bound_assumption} then 
\begin{equation*}
	\int_{\bb R^n}\log(1 + |y|)I_\mu[e^{\varrho u}](y)e^{\varrho u(y)}\; \d y
	< \infty. 
\end{equation*}
\end{lemma}
\begin{proof}
Let $\epsilon>0$ satisfy \eqref{eq:Imu_decay_epsilon_choice}. Lemmata \ref{lemma:Kf_upper_decay_estimate} and \ref{lemma:Imu_initial_decay} guarantee the existence of $R = R(\epsilon)$ for which that both \eqref{eq:Kf_upper_decay_estimate} and \eqref{eq:Imu_initial_decay} hold on $\bb R^n\setminus B_R$. Fix any such $R$. Evidently, 
\begin{equation*}
	\int_{B_R}\log(1 + |y|)I_\mu[e^{\varrho u}](y)e^{\varrho u(y)}\; \d y
	\leq 2\log R \|I_\mu[e^{\varrho u}]e^{\varrho u}\|_{L^1(\bb R^n)},
\end{equation*}
so to complete the proof it suffices to show that
\begin{equation}
\label{eq:log_Imu_integral}
	\int_{\bb R^n\setminus B_R}\log(1 + |y|)I_\mu[e^{\varrho u}](y)e^{\varrho u(y)}\; \d y
	< \infty. 
\end{equation}
To do so, observe that assumption \eqref{eq:V_lower_bound_assumption} and the choice of $\epsilon$ guarantee that
\begin{equation*}
	n 
	< \varrho\left(\frac{2V}{|\bb S^n|} - \epsilon\right)
	+ \min\left\{\mu, -n + \mu + \varrho\left(\frac{2V}{|\bb S^n|} - \epsilon\right)\right\}.
\end{equation*}
In view of this inequality and the fact that \eqref{eq:Kf_upper_decay_estimate} and \eqref{eq:Imu_initial_decay} combine to give the estimate
\begin{equation*}
	I_\mu[e^{\varrho u}](y)e^{\varrho u(y)}
	\leq C|y|^{-\varrho\left(\frac{2V}{|\bb S^n|} - \epsilon\right)}\left(|y|^{-\mu} + |y|^{n - \mu - \varrho\left(\frac{2V}{|\bb S^n|} - \epsilon\right)}\right)
\end{equation*}
for $y\in \bb R^n\setminus B_R$, we find that \eqref{eq:log_Imu_integral} is satisfied.
\end{proof}
Now we have a new (better) integral representation for solutions to \eqref{eq:entire_nonlocal}. 
\begin{lemma}
\label{lemma:improved_integral_rep}
 If $u$ is a distributional solution to \eqref{eq:entire_nonlocal} for which \eqref{eq:V_lower_bound_assumption} holds then there is a polynomial $p$ of degree not exceeding $n-1$ that is bounded above for which
\begin{equation*}
	u(x)
	= (n - 1)!c_n\int_{\bb R^n}\log\left(\frac{1}{|x- y|}\right)I_\mu[e^{\varrho u}](y)e^{\varrho u(y)}\; \d y +p(x). 
\end{equation*}
\end{lemma}
\begin{proof}
Set $F = I_\mu[e^{\varrho u}]e^{\varrho u}$ and observe that Lemma \ref{lemma:remove_log_1_y} guarantees $\log(1 + |\cdot|)F\in L^1(\bb R^n)$. Therefore, using the integral representation for $u$ in Lemma \ref{lemma:polynomial_part_bounded_above}, there is a polynomial $p$ of degree not exceeding $n - 1$ that is bounded above for which
\begin{equation*}
\begin{split}
	\frac{u(x)}{(n -1)!c_n}
	& = \int_{\bb R^n}\log\left(\frac{1 + |y|}{|x - y|}\right)F(y)\; \d y + p(y)\\
	& = \int_{\bb R^n}\log\left(\frac{1}{|x - y|}\right)F(y)\; \d y + \tilde p(y), 
\end{split}
\end{equation*}
where $\tilde p = p + \|\log(1 + |\cdot|)F\|_{L^1(\bb R^n)}$. Since $p$ is bounded above and has degree not exceeding $n - 1$, we find that $\tilde p$ enjoys these same properties. 
\end{proof}
The following corollary follows immediately from Lemma \ref{lemma:improved_integral_rep}. 
\begin{coro}
\label{coro:best_representation}
If in addition to the hypotheses of Lemma \ref{lemma:improved_integral_rep} it is assumed that $u(x) = \circ(|x|^2)$ as $|x|\to\infty$, then there is a real constant $b$ for which 
\begin{equation}
\label{eq:best_representation}
	u(x)
	= (n - 1)!c_n\int_{\bb R^n}\log\left(\frac{1}{|x- y|}\right)I_\mu[e^{\varrho u}](y)e^{\varrho u(y)}\; \d y +b. 
\end{equation}
\end{coro}
Having the integral representation \eqref{eq:best_representation} for solutions to \eqref{eq:entire_nonlocal} satisfying both \eqref{eq:V_lower_bound_assumption} and \eqref{eq:asymtpotic_growth_assumption} allows us to deduce further decay and symmetry properties of such functions. For $(x_0, \sigma)\in \bb R^n\times(0, \infty)$, we define the logarithmic Kelvin transform of a function $u:\bb R^n\to \bb R$ by 
\begin{equation*}
	u_{x_0, \sigma}(x)
	= u(x^{x_0, \sigma}) + 2\log\frac{\sigma}{|x - x_0|}, 
\end{equation*}
where
\begin{equation*}
	x^{x_0, \sigma}
	= x_0 + \frac{\sigma^2(x - x_0)}{|x- x_0|^2}
\end{equation*}
is the inversion of $x\in \bb R^n$ about $\bdy B_\sigma(x_0)$. Deduction of decay and symmetry properties of solutions to \eqref{eq:entire_nonlocal} relies both on  comparisons of $u$ with $u_{x_0, \sigma}$ and on comparisons of $I_\mu[e^{\varrho u}]$ with $I_\mu[e^{\varrho u_{x_0, \sigma}}]$. Lemmata \ref{lemma:Imu_difference} and \ref{lemma:u_difference} below provide integral representations that facilitate these comparisons. 
\begin{lemma}
\label{lemma:Imu_difference}
If $u: \bb R^n\to \bb R$ satisfies $e^u\in L^n(\bb R^n)$ then for any $(x_0, \sigma)\in \bb R^n\times (0, \infty)$ there holds
\begin{equation}
\label{eq:Imu_difference_integral}
	I_\mu[e^{\varrho u_{x_0, \sigma}}](x) - I_\mu[e^{\varrho u}](x)
	= \int_{\bb R^n\setminus B_\sigma(x_0)}\mc K(x_0, \sigma; x, y)\left(e^{\varrho u_{x_0, \sigma}(y)} - e^{\varrho u(y)}\right)\; \d y, 
\end{equation}
where equality is understood in the sense of $L^{2n/\mu}(\bb R^n)$ and 
\begin{equation}
\label{eq:Imu_difference_integral_kernel}
	\mc K(x_0, \sigma; x, y)
	= \frac{1}{|x- y|^\mu} - \frac{1}{|x - y^{x_0, \sigma}|^\mu}\left(\frac\sigma{|y - x_0|}\right)^\mu
\end{equation}
satisfies $K(x_0, \sigma; x, y)>0$ whenever $x, y\in \bb R^n\setminus \overline B_\sigma(x_0)$. 
\end{lemma}
\begin{remark}
If in addition to the hypotheses of Lemma \ref{lemma:Imu_difference} $u$ is assumed to be locally bounded above (which in view of Proposition \ref{prop:distributional_solutions_smooth} is guaranteed whenever $u$ is a solution to \eqref{eq:entire_nonlocal}) then equation \eqref{eq:Imu_difference_integral} may be understood in the pointwise sense. 
\end{remark}
\begin{proof}[Proof of Lemma \ref{lemma:Imu_difference}]
Using the change of variable $y = z^{x_0, \sigma}$, $\d y = \left(\frac{\sigma}{|z- x_0|}\right)^{2n}\; \d z$ we have 
\begin{equation*}
\begin{split}
	\int_{B_\sigma(x_0)}\frac{e^{\varrho u(y)}}{|x- y|^\mu}\; \d y
	& 
	= \int_{\bb R^n\setminus B_\sigma(x_0)}\frac{e^{\varrho u_{x_0, \sigma}(z)}}{|x - z^{x_0, \sigma}|^\mu}\left(\frac{\sigma}{|z- x_0|}\right)^{\mu}\; \d z. 
\end{split}
\end{equation*}
Therefore, 
\begin{equation}
\label{eq:Imu_Bsigma_split}
\begin{split}
	I_\mu[e^{\varrho u}](x)
	& 
	\; \int_{\bb R^n\setminus B_\sigma(x_0)}\frac{e^{\varrho u_{x_0, \sigma}(y)}}{|x - y^{x_0, \sigma}|^\mu}\left(\frac{\sigma}{|y- x_0|}\right)^{\mu}\; \d y\\
	& + \int_{\bb R^n\setminus B_\sigma(x_0)}\frac{e^{\varrho u(y)}}{|x- y|^\mu}\; \d y. 
\end{split}
\end{equation}
For any $x, y\in \bb R^n \setminus\{x_0\}$ we have both
\begin{equation}
\label{eq:inversion_equalities}
	|x^{x_0, \sigma} - y^{x_0, \sigma}|
	= \frac{\sigma^2|x - y|}{|x- x_0||y - x_0|}
	\quad \text{ and }\quad
	\frac{|x^{x_0, \sigma} - y|}{|x - y^{x_0, \sigma}|}
	= \frac{|y- x_0|}{|x - x_0|}.
\end{equation}
Using the change of variable $z\mapsto z^{x_0, \sigma}$ and the first of these equalities shows that
\begin{equation}
\label{eq:Imu_erhou_inversion}
	I_\mu[e^{\varrho u_{x_0, \sigma}}](x)
	= \left(\frac\sigma{|x - x_0|}\right)^\mu I_\mu[e^{\varrho u}](x^{x_0, \sigma}). 
\end{equation}
In view of \eqref{eq:Imu_erhou_inversion}, upon evaluating \eqref{eq:Imu_Bsigma_split} at $x^{x_0, \sigma}$  we obtain 
\begin{equation}
\label{eq:inverted_Imu_Bsigma_split}
\begin{split}
	I_\mu[e^{\varrho u_{x_0, \sigma}}](x)
	& 
	\; \int_{\bb R^n\setminus B_\sigma(x_0)}\frac{e^{\varrho u_{x_0,\sigma}(y)}}{|x- y|^\mu}\; \d y\\
	& + \int_{\bb R^n\setminus B_\sigma(x_0)}\frac{e^{\varrho u(y)}}{|x - y^{x_0, \sigma}|^\mu}\left(\frac\sigma{|y - x_0|}\right)^\mu\; \d y. 
\end{split}
\end{equation}
Subtracting \eqref{eq:Imu_Bsigma_split} from \eqref{eq:inverted_Imu_Bsigma_split} yields \eqref{eq:Imu_difference_integral}. 
The inequality $\mc K(x_0, \sigma; x, y)> 0$ for $x, y\in \bb R^n\setminus \overline B_\sigma(x_0)$ follows from the equality
\begin{equation}
\label{eq:implies_positive_kernels}
	\left(\frac{|y - x_0|}{\sigma}\right)^2|x - y^{x_0, \sigma}|^2 - |x- y|^2
	= \frac{1}{\sigma^2}(|x- x_0|^2 - \sigma^2)(|y- x_0|^2 - \sigma^2). 
\end{equation}
\end{proof}
\begin{lemma}
\label{lemma:u_difference}
If $u$ is a distributional solution to \eqref{eq:entire_nonlocal} for which both  \eqref{eq:V_lower_bound_assumption} and \eqref{eq:asymtpotic_growth_assumption} are satisfied, then for every $(x_0, \sigma)\in \bb R^n\times (0, \infty)$ we have
\begin{equation}
\label{eq:w_x0_sigma_expression}
\begin{split}
	\lefteqn{\frac{u_{x_0, \sigma}(x) - u(x)}{(n - 1)!c_n} + (V - |\bb S^n|)\log\frac{\sigma}{|x- x_0|}}\\
	& = \int_{\bb R^n\setminus B_\sigma(x_0)}\mc L(x_0, \sigma; x, y)\left(I_\mu[e^{\varrho u_{x_0, \sigma}}](y)e^{\varrho u_{x_0, \sigma}(y)} - I_\mu[e^{\varrho u}](y) e^{\varrho u(y)}\right)\; \d y,  
\end{split}
\end{equation}
where $\mc L(x_0, \sigma; x, y)$ is defined by 
\begin{equation*}
	\mc L(x_0, \sigma; x, y)
	= \log\frac{|y - x_0||x - y^{x_0, \sigma}|}{\sigma|x - y|}. 
\end{equation*}
Moreover, $\mc L(x_0, \sigma; x, y)> 0$ for $x, y\in \bb R^n\setminus B_\sigma(x_0)$. 
\end{lemma}
\begin{proof}
Fix $(x_0, \sigma)\in \bb R^n\times (0, \infty)$. Using the change of variable $y\mapsto y^{x_0, \sigma}$ and equation \eqref{eq:Imu_erhou_inversion} we have
\begin{equation}
\label{eq:V_split}
\begin{split}
	V
	= & \; \left(\int_{B_\sigma(x_0)} + \int_{\bb R^n\setminus B_\sigma(x_0)}\right)I_\mu[e^{\varrho u}](y) e^{\varrho u(y)}\; \d y\\
	= & \; \int_{\bb R^n\setminus B_\sigma(x_0)}I_\mu[e^{\varrho u}](y^{x_0, \sigma})e^{\varrho u_{x_0, \sigma}(y)}\left(\frac{\sigma}{|y- x_0|}\right)^\mu\; \d y\\
	&  + \int_{\bb R^n\setminus B_\sigma(x_0)}I_\mu[e^{\varrho u}](y) e^{\varrho u(y)}\; \d y\\
	= & \; \int_{\bb R^n\setminus B_\sigma(x_0)}I_\mu[e^{\varrho u_{x_0, \sigma}}](y)e^{\varrho u_{x_0, \sigma}(y)}\; \d y + \int_{\bb R^n\setminus B_\sigma(x_0)}I_\mu[e^{\varrho u}](y) e^{\varrho u(y)}\; \d y. 
\end{split}
\end{equation}
Using the same change of variable and \eqref{eq:Imu_erhou_inversion} we have
\begin{equation*}
\begin{split}
	\int_{B_\sigma(x_0)}& \log\left(\frac{1}{|x - y|}\right) I_\mu[e^{\varrho u}](y)e^{\varrho u(y)}\; \d y\\
	& 
	= \int_{\bb R^n\setminus B_\sigma(x_0)}\log\left(\frac{1}{|x - y^{x_0, \sigma}|}\right) I_\mu[e^{\varrho u}](y^{x_0, \sigma})e^{\varrho u_{x_0, \sigma}(y)}\; \left(\frac\sigma{|y - x_0|}\right)^{\mu}\; \d y\\
	& = \int_{\bb R^n\setminus B_\sigma(x_0)}\log\left(\frac{1}{|x - y^{x_0, \sigma}|}\right) I_\mu[e^{\varrho u_{x_0, \sigma}}](y)e^{\varrho u_{x_0, \sigma}(y)}\; \d y. 
\end{split}
\end{equation*}
In view of this equality and Corollary \ref{coro:best_representation}, for each $x\in \bb R^n$ we have
\begin{equation}
\label{eq:ux_split}
\begin{split}
	\frac{u(x) - b}{(n - 1)!c_n}
	= & \; \left(\int_{\bb R^n\setminus B_\sigma(x_0)}+\int_{B_\sigma(x_0)}\right)\log\left(\frac 1{|x- y|}\right)I_\mu[e^{\varrho u}](y)e^{\varrho u(y)}\; \d y\\
	= & \; \int_{\bb R^n\setminus B_\sigma(x_0)}\log\left(\frac 1{|x- y|}\right)I_\mu[e^{\varrho u}](y)e^{\varrho u(y)}\; \d y\\
	& + \int_{\bb R^n\setminus B_\sigma(x_0)}\log\left(\frac{1}{|x - y^{x_0, \sigma}|}\right) I_\mu[e^{\varrho u_{x_0, \sigma}}](y)e^{\varrho u_{x_0, \sigma}(y)}\; \d y.
\end{split}
\end{equation}
Evaluation of this equality at $x^{x_0, \sigma}$ and in view of equations \eqref{eq:inversion_equalities} we have
\begin{equation}
\label{eq:ux_inverted_split}
\begin{split}
	\frac{u(x^{x_0, \sigma}) - b}{(n - 1)!c_n}
	& 
	\; \int_{\bb R^n\setminus B_\sigma(x_0)}\log\left(\frac{|x -x_0|}{|y - x_0||x - y^{x_0, \sigma}|}\right)I_\mu[e^{\varrho u}](y)e^{\varrho u(y)}\; \d y\\
	& + \int_{\bb R^n\setminus B_\sigma(x_0)}\log\left(\frac{|x - x_0||y - x_0|}{\sigma^2 |x - y|}\right) I_\mu[e^{\varrho u_{x_0, \sigma}}](y)e^{\varrho u_{x_0, \sigma}(y)}\; \d y.
\end{split}
\end{equation}
Now using \eqref{eq:V_split}, \eqref{eq:ux_split}, and \eqref{eq:ux_inverted_split}, for every $x\in \bb R^n\setminus \{x_0\}$ we have
\begin{equation}
\label{eq:w_x0sigma_initial_split}
\begin{split}
	\lefteqn{\frac{u_{x_0, \sigma}(x) - u(x)}{(n - 1)!c_n} - \frac{2}{(n - 1)!c_n}\log\frac\sigma{|x - x_0|}}\\
	& = \frac{u(x^{x_0, \sigma})-u(x)}{(n - 1)!c_n}\\
	& = \int_{\bb R^n\setminus B_\sigma(x_0)}\log\left(\frac{|x -x_0||x - y|}{|y - x_0||x - y^{x_0, \sigma}|}\right)I_\mu[e^{\varrho u}](y)e^{\varrho u(y)}\; \d y\\
	& \quad + \int_{\bb R^n\setminus B_\sigma(x_0)}\log\left(\frac{|x - x_0||y - x_0||x - y^{x_0, \sigma}|}{\sigma^2 |x - y|}\right) I_\mu[e^{\varrho u_{x_0, \sigma}}](y)e^{\varrho u_{x_0, \sigma}(y)}\; \d y.
\end{split}
\end{equation}
In view of \eqref{eq:cn} and \eqref{eq:V_split} adding $V\log\frac{\sigma}{|x - x_0|}$ to both sides of \eqref{eq:w_x0sigma_initial_split} yields \eqref{eq:w_x0_sigma_expression}.
The inequality $\mc L(x_0, \sigma; x, y)> 0$ for $x, y\in \bb R^n\setminus B_\sigma(x_0)$ follows from \eqref{eq:implies_positive_kernels}. 
\end{proof}
\section{The Precise Decay Rate of Solutions}
\label{s:precise_decay}
This section is devoted to the computing the decay rate of solutions $u$ to \eqref{eq:entire_nonlocal} that satisfy both \eqref{eq:V_lower_bound_assumption} and \eqref{eq:asymtpotic_growth_assumption}. In view of Theorem \ref{theorem:asymptotic}, these assumptions imply that the value of $V$ dictates the decay rate. The following proposition is the main result of this section. 
\begin{prop}
\label{prop:precise_decay}
If $u$ is a distributional solution to \eqref{eq:entire_nonlocal} for which \eqref{eq:V_lower_bound_assumption} and \eqref{eq:asymtpotic_growth_assumption} are satisfied, then $V=|\bb S^n|$.  
\end{prop}
The proof Proposition \ref{prop:precise_decay} will be accomplished by separately establishing the inequalities $V\geq |\bb S^n|$ and $V\leq |\bb S^n|$. These inequalities correspond to ruling out slow decay and ruling out fast decay respectively. They will be proven in Subsections \ref{ss:rule_out_slow_decay} and \ref{ss:rule_out_fast_decay} respectively via suitable applications of the method of moving spheres. For $S\subset \bb R^n$ and $(x_0, \sigma)\in \bb R^n\times (0, \infty)$ we use the notation 
\begin{equation*}
	S^{x_0, \sigma} = \{x^{x_0, \sigma}: x\in S\} 
\end{equation*}
for the inversion of $S$ about $\bdy B_\sigma(x_0)$. 
\subsection{Ruling out Slow Decay}
\label{ss:rule_out_slow_decay}
The purpose of this subsection is to establish the following lemma. 
\begin{lemma}
\label{lemma:rule_out_slow_decay}
If $u$ is a distributional solution to \eqref{eq:entire_nonlocal} for which \eqref{eq:V_lower_bound_assumption} and \eqref{eq:asymtpotic_growth_assumption} are satisfied, then $V\geq |\bb S^n|$.  
\end{lemma}
Lemma \ref{lemma:rule_out_slow_decay} will be established with the aid of a series of lemmata. For $(x_0, \sigma)\in \bb R^n\times (0, \infty)$ we define the sets
\begin{equation*}
\begin{split}
	\mc P(x_0, \sigma)& = \{x\in \bb R^n\setminus \overline B_\sigma(x_0): u(x)< u_{x_0, \sigma}(x)\}\\
	\mc Q(x_0, \sigma)& = \{x\in \bb R^n\setminus \overline B_\sigma(x_0): I_\mu[e^{\varrho u}](x)< I_\mu[e^{\varrho u_{x_0, \sigma}}](x)\}. 
\end{split}
\end{equation*}
\begin{lemma}
\label{lemma:Qnonempty}
Let $u: \bb R^n\to \bb R$ satisfy $e^u\in L^n(\bb R^n)$. If $(x_0, \sigma)\in \bb R^n\times (0, \infty)$ with $\mc Q(x_0, \sigma)\neq \emptyset$ then $|\mc P(x_0, \sigma)|> 0$. Consequently, if $|\mc P(x_0, \sigma)| + |\mc Q(x_0, \sigma)|> 0$ then $|\mc P(x_0, \sigma)|> 0$. 
\end{lemma}
\begin{proof}
For any $x\in \bb R^n\setminus B_\sigma(x_0)$, using Lemma \ref{lemma:Imu_difference} and the inequality $\mc K(x_0, \sigma; x, y)\leq |x- y|^{-\mu}$ we have
\begin{equation*}
\begin{split}
	I_\mu[e^{\varrho u_{x_0, \sigma}}](x) - I_\mu[e^{\varrho u}](x)
	& = \int_{\bb R^n\setminus B_\sigma(x_0)}\mc K(x_0, \sigma; x, y)\left(e^{\varrho u_{x_0, \sigma}(y)} - e^{\varrho u(y)}\right)\; \d y\\
	& \leq  \int_{\mc P(x_0, \sigma)}\mc K(x_0, \sigma; x, y)\left(e^{\varrho u_{x_0, \sigma}(y)} - e^{\varrho u(y)}\right)\; \d y\\
	& \leq \int_{\mc P(x_0, \sigma)}\frac{e^{\varrho u_{x_0, \sigma}(y)} - e^{\varrho u(y)}}{|x- y|^\mu}\; \d y\\
	& = I_\mu\left[\left(e^{\varrho u_{x_0, \sigma}} - e^{\varrho u}\right)\chi_{\mc P(x_0, \sigma)}\right](x). 
\end{split}
\end{equation*}
Therefore, if $x\in \mc Q(x_0, \sigma)$ the we obtain 
\begin{equation*}
\begin{split}
	0 
	& < I_\mu[e^{\varrho u_{x_0, \sigma}}](x) - I_\mu[e^{\varrho u}](x)\\
	&\leq I_\mu\left[\left(e^{\varrho u_{x_0, \sigma}} - e^{\varrho u}\right)\chi_{\mc P(x_0, \sigma)}\right](x)
\end{split}
\end{equation*}
and hence $|\mc P(x_0,\sigma)|> 0$. 
\end{proof}
Lemmata \ref{lemma:MS_start_main_estimate} and \ref{lemma:moving_spheres_can_start} that follow guarantee that if $V\leq |\bb S^n|$ then the moving sphere process centered at any $x_0\in \bb R^n$ can start. For the purposes of ruling out slow decay we only need the versions of these lemmata corresponding to $V< |\bb S^n|$. The versions of these lemmata corresponding to $V = |\bb S^n|$ will be used in Section \ref{s:classification} to prove Theorem \ref{theorem:classification}.

\begin{lemma}
\label{lemma:MS_start_main_estimate}
Let $u$ be a distributional solution to \eqref{eq:entire_nonlocal} satisfying both $\frac{n - \mu}{2n - \mu}< \frac{V}{|\bb S^n|}\leq 1$ and \eqref{eq:asymtpotic_growth_assumption}. For all $x_0\in \bb R^n$ and all $\Sigma\in (0, \infty)$ there is a constant $C>0$ such that for all $\sigma\in (0, \Sigma]$ there holds
\begin{equation}
\label{eq:MS_start_main_estimate}
\begin{split}
	\lefteqn{\|e^{\varrho u_{x_0, \sigma}} - e^{\varrho u}\|_{L^{n/\varrho}(\mc P(x_0, \sigma))}}\\
	& \leq C\|1 + |\log|\cdot - x_0||\|_{L^{n/\varrho}(\mc P(x_0, \sigma)^{x_0, \sigma})}\|e^{\varrho u_{x_0, \sigma}} - e^{\varrho u}\|_{L^{n/\varrho}(\mc P(x_0, \sigma))}. 
\end{split}
\end{equation}
We emphasize that $C$ is independent of $\sigma\in (0, \Sigma]$. 
\end{lemma}
\begin{proof}
Fix $(x_0, \Sigma)\in \bb R^n\times(0, \infty)$. Define 
\begin{equation}
\label{eq:L_bar}
	\bar{\mc L}(x_0, \sigma; x, y)
	= \log\frac{2|x - x_0||y - x_0|}{\sigma|x- y|}
\end{equation}	
so that $\mc L(x_0, \sigma; x, y)\leq \bar{\mc L}(x_0, \sigma; x, y)$ whenever $x, y\in \bb R^n\setminus B_\sigma(x_0)$. 
Setting 
\begin{equation*}
\begin{split}
	H_1(x_0, \sigma; x) 
	& = \int_{\mc Q(x_0, \sigma)}\bar{\mc L}(x_0, \sigma; x,y)I_\mu[(e^{\varrho u_{x_0,\sigma}}- e^{\varrho u})\chi_{\mc P(x_0, \sigma)}](y)e^{\varrho u_{x_0, \sigma}(y)}\; \d y\\
	H_2(x_0, \sigma; x)
	& = \int_{\mc P(x_0, \sigma)}\bar {\mc L}(x_0, \sigma; x,y)I_\mu[e^{\varrho u}](y)\left(e^{\varrho u_{x_0, \sigma}(y)} - e^{\varrho u(y)}\right)\; \d y, 
\end{split}
\end{equation*}
for any $x\in \mc P(x_0, \sigma)$ the assumption $V\leq |\bb S^n|$, the Mean Value Theorem, and Lemma \ref{lemma:u_difference} give
\begin{equation}
\label{eq:get_H1_H2}
\begin{split}
	0
	< & \; \frac{e^{\varrho u_{x_0, \sigma}(x)} - e^{\varrho u(x)}}{\varrho (n - 1)!c_ne^{\varrho u_{x_0, \sigma}(x)}}\\
	\leq &\; \frac{e^{\varrho u_{x_0, \sigma}(x)} - e^{\varrho u(x)}}{\varrho (n - 1)!c_ne^{\varrho u_{x_0, \sigma}(x)}} + (V - |\bb S^n|)\log\frac{\sigma}{|x - x_0|}\\
	\leq & \;  \int_{\bb R^n\setminus B_\sigma(x_0)}\mc L(x_0, \sigma; x,y)\left(I_\mu[e^{\varrho u_{x_0,\sigma}}](y)e^{\varrho u_{x_0, \sigma}(y)} - I_\mu[e^{\varrho u}](y)e^{\varrho u(y)}\right)\; \d y\\
	\leq & \;  \int_{\mc Q(x_0, \sigma)}\mc L(x_0, \sigma; x,y)\left(I_\mu[e^{\varrho u_{x_0,\sigma}}](y) - I_\mu[e^{\varrho u}](y)\right)e^{\varrho u_{x_0, \sigma}(y)}\; \d y\\
	& + \int_{\mc P(x_0, \sigma)}\mc L(x_0, \sigma; x,y)I_\mu[e^{\varrho u}](y)\left(e^{\varrho u_{x_0, \sigma}(y)} - e^{\varrho u(y)}\right)\; \d y\\
	\leq &\; H_1(x_0, \sigma; x) + H_2(x_0, \sigma; x). 
\end{split}
\end{equation}
According to this estimate and Minkowski's inequality we have
\begin{equation}
\label{eq:erhou_difference_first_bound}
	\|e^{\varrho u_{x_0, \sigma}(x)} - e^{\varrho u(x)}\|_{L^{n/\varrho}(\mc P(x_0,\sigma)}
	\leq C\sum_{i = 1}^2\|e^{\varrho u_{x_0, \sigma}}H_i(x_0, \sigma; \cdot)\|_{L^{n/\varrho}(\mc P(x_0, \sigma))}  
\end{equation}
for some constant $C = C(n)>0$. We proceed to estimate each summand on the right-hand side of this inequality. In the estimates that follows $C$ will denote various constants that may depend on $n$, $\mu$, $\varrho$, $\|e^u\|_{L^n(\bb R^n)}$, $\|e^u\|_{L^\infty(\bb R^n)}$, $\|I_\mu[e^{\varrho u}]\|_{L^\infty(\bb R^n)}$, and $\Sigma$ but are independent of $\sigma\in (0, \Sigma]$. We begin by establishing the pointwise estimates 
\begin{equation}
\label{eq:Hi_pointwise_estimates}
	H_i(x_0, \sigma; x)
	\leq C\left(1 + \log\frac{|x - x_0|}{\sigma}\right)\|e^{\varrho u_{x_0, \sigma}} - e^{\varrho u}\|_{L^{n/\varrho}(\mc P(x_0, \sigma))}
\end{equation}
for $i = 1, 2$ and $x\in \mc P(x_0, \sigma)$. To show that \eqref{eq:Hi_pointwise_estimates} holds for $i = 1$, for $x\in \mc P(x_0, \sigma)$ define
\begin{equation*}
\begin{split}
	E^1(x) & = \mc Q(x_0, \sigma)\setminus B_{2|x - x_0|}(x_0)\\
	E^2(x) & = \left(\mc Q(x_0, \sigma)\cap B_{2|x - x_0|}(x_0)\right)\setminus B_1(x)\\
	E^3(x) & = \mc Q(x_0, \sigma)\cap B_{2|x - x_0|}(x_0)\cap B_1(x)
\end{split}
\end{equation*}
and set 
\begin{equation*}
	H_1^j(x_0, \sigma; x)
	= \int_{E^j(x)}\bar{\mc L}(x_0, \sigma; x,y)I_\mu[(e^{\varrho u_{x_0,\sigma}}- e^{\varrho u})\chi_{\mc P(x_0, \sigma)}](y)e^{\varrho u_{x_0, \sigma}(y)}\; \d y
\end{equation*}
so that 
\begin{equation}
\label{eq:H_1_decompose}
	H_1(x_0, \sigma; x)
	= \sum_{j = 1}^3 H_1^j(x_0, \sigma; x). 
\end{equation}
If $y\in E^1(x)$ then $0<\bar{\mc L}(x_0, \sigma; x, y)\leq \log(\frac{4|x - x_0|}\sigma)$ so H\"older's inequality and the HLS inequality give
\begin{equation}
\label{eq:H11_estimate}
\begin{split}
	H_1^1(x_0, \sigma; x)
	& \leq\log\frac{4|x- x_0|}{\sigma}\int_{\bb R^n}I_\mu[(e^{\varrho u_{x_0, \sigma}} - e^{\varrho u})\chi_{\mc P(x_0, \sigma)}](y)e^{\varrho u_{x_0, \sigma}(y)}\; \d y\\
	& \leq \log\frac{4|x- x_0|}\sigma \|e^{\varrho u_{x_0, \sigma}}\|_{L^{n/\varrho}(\bb R^n)}\|I_\mu[(e^{\varrho u_{x_0, \sigma}}- e^{\varrho u})\chi_{\mc P(x_0, \sigma)}]\|_{L^{2n/\mu}(\bb R^n)}\\
	& \leq C\log\frac{4|x -x_0|}{\sigma}\|e^u\|_{L^n(\bb R^n)}^\varrho\|e^{\varrho u_{x_0,\sigma}} - e^{\varrho u}\|_{L^{n/\varrho}(\mc P(x_0, \sigma))}\\
	& \leq C\left(1 + \log\frac{|x - x_0|}\sigma\right)\|e^{\varrho u_{x_0, \sigma}} - e^{\varrho u}\|_{L^{n/\varrho}(\mc P(x_0, \sigma))}. 
\end{split}
\end{equation}
If $y\in E^2(x)$ then 
\begin{equation*}
	1< \frac{2|x - x_0||y - x_0|}{\sigma|x- y|}
	\leq \frac{4|x - x_0|^2}{\sigma}
	\leq \frac{4\Sigma|x - x_0|^2}{\sigma^2}
\end{equation*}
so 
\begin{equation}
\label{eq:y_in_E2_kernel_estimate}
	0
	< \bar{\mc L}(x_0, \sigma; x, y)
	\leq \log\frac{4\Sigma|x - x_0|^2}{\sigma^2}
	\leq C(\Sigma)\left(1 + \log\frac{|x - x_0|}{\sigma}\right). 
\end{equation}
Therefore, estimating similarly to \eqref{eq:H11_estimate} we have
\begin{equation*}
\begin{split}
	H_1^2(x_0, \sigma; x)
	& \leq C\left(1 + \log\frac{|x - x_0|}{\sigma}\right)\int_{\bb R^n}I_\mu[(e^{\varrho u_{x_0, \sigma}} - e^{\varrho u})\chi_{\mc P(x_0, \sigma)}](y)e^{\varrho u_{x_0, \sigma}(y)}\; \d y\\
	& 
	\leq C\left(1 + \log\frac{|x - x_0|}{\sigma}\right)\|e^{\varrho u_{x_0, \sigma}} - e^{\varrho u}\|_{L^{n/\varrho}(\mc P(x_0, \sigma))}. 
\end{split}
\end{equation*}
If $y\in E^3(x)$ then 
\begin{equation*}
	1
	< \frac{2|x - x_0||y - x_0|}{\sigma|x- y|}
	\leq \frac{4|x -x_0|^2}{\sigma|x- y|}
	\leq \frac{4\Sigma|x - x_0|^2}{\sigma^2|x- y|}, 
\end{equation*}
so 
\begin{equation}
\label{eq:y_in_E3_kernel_estimate}
	0
	< \bar{\mc L}(x_0, \sigma; x, y)
	\leq C\left(1 + \log\frac{|x- x_0|}{\sigma} + \log\frac{1}{|x - y|}\right). 
\end{equation}
Therefore, 
\begin{equation*}
\begin{split}
	H_1^3(x_0, \sigma; x)
	& 
	\; C\left(1 + \log\frac{|x- x_0|}{\sigma}\right)\int_{\bb R^n}I_\mu[(e^{\varrho u_{x_0, \sigma}}- e^{\varrho u})\chi_{\mc P(x_0, \sigma)}](y)e^{\varrho u_{x_0, \sigma}(y)}\; \d y\\
	& + C\int_{B_1(x)\setminus B_\sigma(x_0)}|\log|x - y||I_\mu[(e^{\varrho u_{x_0, \sigma}}-e^{\varrho u})\chi_{\mc P(x_0, \sigma)}](y)e^{\varrho u_{x_0, \sigma}(y)}\; \d y\\
	\leq & \; C\left(1+ \log\frac{|x- x_0|}\sigma\right)\|e^u\|_{L^n(\bb R^n)}^{\varrho} \|e^{\varrho u_{x_0, \sigma}} - e^{\varrho u}\|_{L^{n/\varrho}(\mc P(x_0, \sigma))}\\
	& + C\|e^u\|_{L^\infty(\bb R^n)}\|\log|x - \cdot|\|_{L^{n/\varrho}(B_1(x))}\|e^{\varrho u_{x_0,  \sigma}} - e^{\varrho u}\|_{L^{n/\varrho}(\mc P(x_0, \sigma))}\\
	\leq & \; C\left(1 + \log\frac{|x - x_0|}\sigma\right)\|e^{\varrho u_{x_0, \sigma}} - e^{\varrho u}\|_{L^{n/\varrho}(\mc P(x_0, \sigma))}. 
\end{split}
\end{equation*}
Bringing the estimates for $H_1^j(x_0, \sigma; x)$ ($j = 1, 2, 3$) back to equation \eqref{eq:H_1_decompose} gives \eqref{eq:Hi_pointwise_estimates} for $i = 1$. To show that \eqref{eq:Hi_pointwise_estimates} holds for $i = 2$, for $x\in \mc P(x_0, \sigma)$ define
\begin{equation*}
\begin{split}
	D^1(x) & = \mc P(x_0, \sigma)\setminus B_{2|x - x_0|}(x_0)\\
	D^2(x) & = \left(\mc P(x_0, \sigma)\cap B_{2|x - x_0|}(x_0)\right)\setminus B_1(x)\\
	D^3(x) & = \mc P(x_0, \sigma)\cap B_{2|x - x_0|}(x_0)\cap B_1(x)
\end{split}
\end{equation*}
and set
\begin{equation*}
	H_2^j(x_0, \sigma; x)
	= \int_{D^j(x)}\bar{\mc L}(x_0, \sigma; x, y)I_\mu[e^{\varrho u}](y)(e^{\varrho u_{x_0, \sigma}(y)} - e^{\varrho u(y)})\; \d y
\end{equation*}
so that 
\begin{equation*}
	H_2(x_0, \sigma; x)
	= \sum_{j = 1}^3H_2^j(x_0, \sigma; x). 
\end{equation*}
If $y\in D^1(x)$ then $\bar{\mc L}(x_0, \sigma; x)\leq \log\frac{4|x - x_0|}\sigma$ so H\"older's inequality and the HLS inequality give
\begin{equation*}
\begin{split}
	H_2^1(x_0, \sigma;x)
	& \leq \log\frac{4|x - x_0|}\sigma\int_{D^1(x)}I_\mu[e^{\varrho u}](y)(e^{\varrho u_{x_0, \sigma}(y)} - e^{\varrho u(y)})\; \d y\\
	& \leq \log\frac{4|x - x_0|}\sigma \|I_\mu[e^{\varrho u}]\|_{L^{2n/\mu}(\bb R^n)}\|e^{\varrho u_{x_0, \sigma}} - e^{\varrho u}\|_{L^{n/\rho}(\mc P(x_0, \sigma))}\\
	& \leq C\log\left(1 + \frac{|x - x_0|}{\sigma}\right)\|e^{\varrho u_{x_0, \sigma}} - e^{\varrho u}\|_{L^{n/\rho}(\mc P(x_0, \sigma))}.
\end{split}
\end{equation*}
If $y\in D^2(x)$ then estimate \eqref{eq:y_in_E2_kernel_estimate} holds so 
\begin{equation*}
\begin{split}
	H_2^2(x_0, \sigma; x)
	& \leq C\left(1 + \log\frac{|x - x_0|}\sigma\right)\int_{D^2(x)}I_\mu[e^{\varrho u}](y)(e^{\varrho u_{x_0, \sigma}(y)} - e^{\varrho u(y)})\; \d y\\
	& \leq C\left(1 + \log\frac{|x - x_0|}\sigma\right)\|e^{\varrho u_{x_0, \sigma}}- e^{\varrho u}\|_{L^{n/\varrho}(\mc P(x_0, \sigma))}. 
\end{split}
\end{equation*}
If $y\in D^3(x)$ then estimate \eqref{eq:y_in_E3_kernel_estimate} holds so 
\begin{equation*}
\begin{split}
	\lefteqn{H_2^3(x_0, \sigma; x)}\\
	 & \leq C\int_{D^3(x)}\left(1 + \log\frac{|x - x_0|}\sigma + \log\frac 1{|x- y|}\right)I_\mu[e^{\varrho u}](y)(e^{\varrho u_{x_0,\sigma}(y)} - e^{\varrho u(y)})\; \d y\\
	& \leq C\left(1 + \log\frac{|x - x_0|}\sigma\right)\|e^{\varrho u_{x_0, \sigma}} - e^{\varrho u}\|_{L^{n/\varrho}(\mc P(x_0, \sigma))}\\
	& \quad + C\int_{D^3(x)}|\log|x - y||I_\mu[e^{\varrho u}](y)(e^{\varrho u_{x_0, \sigma}(y)} - e^{\varrho u(y)})\; \d y\\
	& \leq C\left(1 + \log\frac{|x - x_0|}\sigma\right)\|e^{\varrho u_{x_0, \sigma}} - e^{\varrho u}\|_{L^{n/\varrho}(\mc P(x_0, \sigma))}\\
	& \quad + C\|I_\mu[e^{\varrho u}]\|_{L^\infty(\bb R^n)}\|\log|x - \cdot|\|_{L^{2n/\mu}(B_1(x))}\|e^{\varrho u_{x_0, \sigma}} - e^{\varrho u}\|_{L^{n/\varrho}(\mc P(x_0,\sigma))}\\
	& \leq C\left(1 + \log \frac{|x- x_0|}\sigma\right)\|e^{\varrho u_{x_0, \sigma}} - e^{\varrho u}\|_{L^{n/\varrho}(\mc P(x_0, \sigma))}. 
\end{split}
\end{equation*}
Combining the estimates for $H_2^j(x_0, \sigma; x)$ ($j = 1, 2, 3$) gives \eqref{eq:Hi_pointwise_estimates} for $i = 2$. Next, using the change of variable $x\mapsto x^{x_0,\sigma}$ we have
\begin{equation}
\label{eq:507_FE}
\begin{split}
	\lefteqn{\left\|e^{\varrho u_{x_0, \sigma}}(1 + \log(\sigma^{-1}|\cdot - x_0|))\right\|_{L^{n/\varrho}(\mc P(x_0, \sigma))}^{n/\varrho}}\\
	& 
	= \int_{\mc P(x_0, \sigma)^{x_0,  \sigma}}e^{nu(x)}\left(1 +\log\frac\sigma{|x- x_0|}\right)^{n/\varrho}\; \d x\\
	& \leq \|e^u\|_{L^\infty(\bb R^n)}^n\int_{\mc P(x_0, \sigma)^{x_0,  \sigma}}\left(1 +\log\Sigma + |\log|x- x_0||\right)^{n/\varrho}\; \d x\\
	& \leq C\int_{\mc P(x_0, \sigma)^{x_0,  \sigma}}\left(1 +|\log|x- x_0||\right)^{n/\varrho}\; \d x. 
\end{split}
\end{equation}
Therefore, using \eqref{eq:507_FE} and \eqref{eq:Hi_pointwise_estimates}, for each $i\in \{1, 2\}$ we have
\begin{equation*}
\begin{split}
	\lefteqn{\|e^{\varrho u_{x_0, \sigma}}H_i(x_0, \sigma; \cdot)\|_{L^{n/\varrho}(\mc P(x_0, \sigma))}}\\
	& \leq C\|e^{\varrho u_{x_0,\sigma}} - e^{\varrho u}\|_{L^{n/\rho}(\mc P(x_0, \sigma))}\left\|e^{\varrho u_{x_0, \sigma}}(1 + \log(\sigma^{-1}|\cdot - x_0|))\right\|_{L^{n/\varrho}(\mc P(x_0, \sigma))}\\
	& \leq C\|e^{\varrho u_{x_0,\sigma}} - e^{\varrho u}\|_{L^{n/\rho}(\mc P(x_0, \sigma))}\|1 + |\log|\cdot - x_0||\|_{L^{n/\varrho}(\mc P(x_0, \sigma)^{x_0, \sigma})}. 
\end{split}
\end{equation*}
Now returning to inequality \eqref{eq:erhou_difference_first_bound} we obtain the asserted estimate. 
\end{proof}
\begin{lemma}
\label{lemma:moving_spheres_can_start}
Let $u$ be a distributional solution to \eqref{eq:entire_nonlocal} that satisfies both $\frac{n - \mu}{2n - \mu}< \frac{V}{|\bb S^n|}\leq 1$ and \eqref{eq:asymtpotic_growth_assumption}. For every $x_0\in \bb R^n$ there is $\tilde \sigma(x_0)\in (0, 1)$ such that both 
\begin{equation}
\label{eq:uwins}
	u\geq u_{x_0, \sigma}
	\quad \text{ in } \bb R^n\setminus B_\sigma(x_0)
\end{equation}
and
\begin{equation}
\label{eq:Imu_uwins}
	I_\mu[e^{\varrho u}]\geq I_\mu[e^{\varrho u_{x_0, \sigma}}]
	\quad \text{ in } \bb R^n\setminus B_\sigma(x_0)
\end{equation}
whenever $\sigma\in (0, \tilde \sigma(x_0))$. 
\end{lemma}
\begin{proof}
Fix $x_0\in \bb R^n$ and apply Lemma \ref{lemma:MS_start_main_estimate} with $\Sigma = 1$ to obtain a constant $C>0$ (independent of $\sigma\in (0, 1]$) such that estimate \eqref{eq:MS_start_main_estimate} holds for all $\sigma\in (0, 1]$. Since $\mc P(x_0, \sigma)^{x_0,\sigma}\subset B_\sigma(x_0)\subset B_1(x_0)$ and since $1 + |\log|\cdot - x_0||\in L^{n/\varrho}(B_1(x_0))$ we have 
\begin{equation*}
	\lim_{\sigma\to 0^+}\|1 + |\log|\cdot - x_0||\|_{L^{n/\varrho}(\mc P(x_0, \sigma)^{x_0, \sigma})}
	= 0. 
\end{equation*}
Combining this equality with estimate \eqref{eq:MS_start_main_estimate}, we deduce the existence of $\tilde\sigma\in (0, 1)$ such that 
\begin{equation*}
	\|e^{\varrho u_{x_0, \sigma}} - e^{\varrho u}\|_{L^{n/\varrho}(\mc P(x_0, \sigma))}
	\leq \frac 12\|e^{\varrho u_{x_0, \sigma}} - e^{\varrho u}\|_{L^{n/\varrho}(\mc P(x_0, \sigma))}
\end{equation*}
whenever $\sigma\in (0, \tilde \sigma)$. This inequality implies that for any such $\tilde \sigma$ and for any $\sigma\in (0, \tilde\sigma)$ we have $|\mc P(x_0, \sigma)| = 0$ and so Lemma \ref{lemma:Qnonempty} gives $|\mc Q(x_0, \sigma)| = 0$. 
\end{proof}
Under the hypotheses of Lemma \ref{lemma:moving_spheres_can_start}, for every $x_0\in \bb R^n$ the quantity
\begin{equation}
\label{eq:symmetry_radius}
	\Sigma(x_0)
	= \sup\{\beta> 0: \text{ both \eqref{eq:uwins} and \eqref{eq:Imu_uwins} hold whenever $\sigma\in (0, \beta)$}\}
\end{equation}
is well-defined in $(0, +\infty]$. 
\begin{lemma}
\label{lemma:critical_radius_soft_inequalities}
Let $u$ be a distributional solution of \eqref{eq:entire_nonlocal} that satisfies both $\frac{n - \mu}{2n - \mu}< \frac V{|\bb S^n|}\leq 1$ and \eqref{eq:asymtpotic_growth_assumption}. If $x_0\in \bb R^n$ with $\Sigma(x_0)< \infty$ then $|\mc P(x_0, \Sigma(x_0))| + |\mc Q(x_0, \Sigma(x_0))| = 0$. 
\end{lemma}
\begin{proof}
Suppose $x_0\in \bb R^n$ with $\Sigma(x_0)< \infty$ and for ease of notation write $\Sigma$ in place of $\Sigma(x_0)$. In view of Lemma \ref{lemma:Qnonempty}, to complete the proof it suffices to show that $|\mc P(x_0, \Sigma)| = 0$. Proceeding by way of contradiction, suppose $|\mc P(x_0, \Sigma)| > 0$. Choose $\delta>0$ and $R\gg 1$ such that
\begin{equation*}
	3|\mc P(x_0, \Sigma)|
	\leq 4|\{x\in A(\Sigma, R): u_{x_0, \Sigma}(x) - u(x)> \delta\}|. 
\end{equation*}
If $x\in A(\Sigma, R)$ with $u_{x_0, \Sigma} - u(x)> \delta$ then for any $\sigma\in (\Sigma/2, \Sigma)$ that is sufficiently close to $\Sigma$ 
we have 
\begin{equation*}
\begin{split}
	0
	& \geq u_{x_0, \sigma}(x) - u(x)\\
	& \geq u_{x_0, \sigma}(x) - u_{x_0, \Sigma}(x) + \delta\\
	& = u(x^{x_0,\sigma}) - u(x^{x_0, \Sigma}) + 2\log\frac \sigma\Sigma + \delta\\
	& \geq u(x^{x_0,\sigma}) - u(x^{x_0, \Sigma}) + \frac\delta2. 
\end{split}
\end{equation*}
Therefore, 
\begin{equation}
\label{eq:wait}
	3|\mc P(x_0, \Sigma)|
	\leq 4|\{x\in A(\Sigma, R): \frac \delta 2\leq u(x^{x_0, \Sigma}) - u(x^{x_0, \sigma})\}|. 
\end{equation}
On the other hand, since $u$ is uniformly continuous on $A((4R)^{-1}\Sigma^2, \Sigma)$, by choosing $\sigma\in (\Sigma/2, \Sigma)$ closer to $\Sigma$ if necessary we have
\begin{equation*}
	|\{x\in A(\Sigma, R): \frac \delta 2\leq u(x^{x_0, \Sigma}) - u(x^{x_0, \sigma})\}|
	< \frac 12|\mc P(x_0, \Sigma)|
\end{equation*}
which contradicts \eqref{eq:wait}. 
\end{proof}
\begin{lemma}
Let $u$ be a distributional solution of \eqref{eq:entire_nonlocal} that satisfies both $\frac{n - \mu}{2n - \mu}< \frac V{|\bb S^n|}< 1$ and \eqref{eq:asymtpotic_growth_assumption}. If $x_0\in \bb R^n$ with $\Sigma(x_0)< \infty$ then both of the following inequalities hold:  
\begin{equation}
\label{eq:critical_radius_strict_inequality}
	u(x)> u_{x_0, \Sigma(x_0)}(x)
	\quad \text{for every $x\in \bb R^n\setminus \overline B_{\Sigma(x_0)}$}
\end{equation}
and
\begin{equation}
\label{eq:critical_radius_Imu_strict_inequality}
	I_\mu[e^{\varrho u}](x)> I_\mu[e^{\varrho u_{x_0, \Sigma(x_0)}}](x)
	\quad \text{for every $x\in \bb R^n\setminus \overline B_{\Sigma(x_0)}$}.
\end{equation}
\end{lemma}
\begin{proof}
Suppose $x_0\in \bb R^n$ satisfies $\Sigma(x_0)< \infty$. For ease of notation, in the remainder of the proof we write $\Sigma$ in place of $\Sigma(x_0)$. Lemma \ref{lemma:critical_radius_soft_inequalities} guarantees that both $u\geq u_{x_0, \Sigma}$ and $I_\mu[e^{\varrho u}]\geq I_\mu[e^{\varrho u_{x_0, \Sigma}}]$ in $\bb R^n\setminus\overline B_{\Sigma}(x_0)$. Therefore, for any $x\in \bb R^n\setminus \overline B_{\Sigma}(x_0)$ Lemma \ref{lemma:u_difference} gives
\begin{equation*}
\begin{split}
	\lefteqn{\frac{u_{x_0, \Sigma}(x) - u(x)}{c_n(n - 1)!}}\\
	& < \frac{u_{x_0, \Sigma}(x) - u(x)}{c_n(n - 1)!}- (|\bb S^n| - V)\log\frac{\Sigma}{|x- x_0|}\\
	& = \int_{\bb R^n\setminus B_{\Sigma}(x_0)}\mc L(x_0, \Sigma; x, y)\left(I_\mu[e^{\varrho u_{x_0, \Sigma}}](y)e^{\varrho u_{x_0, \Sigma}(y)} - I_\mu[e^{\varrho u}](y)e^{\varrho u(y)}\right)\; \d y\\
	&\leq 0, 
\end{split}
\end{equation*}
thereby establishing \eqref{eq:critical_radius_strict_inequality}. The second of the asserted inequalities now follows from Lemma \ref{lemma:Imu_difference} and inequality \eqref{eq:critical_radius_strict_inequality}. 
\end{proof}
\begin{lemma}
\label{lemma:bad_set_vanishing_measure}
Let $u$ be a distributional solution of \eqref{eq:entire_nonlocal} that satisfies both $\frac{n - \mu}{2n - \mu}< \frac V{|\bb S^n|}< 1$ and \eqref{eq:asymtpotic_growth_assumption}. If $x_0\in \bb R^n$ with $\Sigma(x_0)< \infty$ then 
\begin{equation*}
	\lim_{\sigma\to \Sigma(x_0)^+}|\mc P(x_0, \sigma)^{x_0, \sigma}|
	= 0. 
\end{equation*}
\end{lemma}
\begin{remark}
\label{remark:on_lemma bad_set_vanishing_measure}
An inspection of the proof of Lemma \ref{lemma:bad_set_vanishing_measure} that follows shows that the conclusion of the lemma holds in the case $V = |\bb S^n|$ provided that, in addition to the stated hypotheses, inequalities \eqref{eq:critical_radius_strict_inequality} and \eqref{eq:critical_radius_Imu_strict_inequality} are assumed to hold. This fact will be used later in the proof of Lemma \ref{lemma:symmetric_about_symmetry_sphere}. 
\end{remark}
\begin{proof}[Proof of Lemma \ref{lemma:bad_set_vanishing_measure}]
Suppose $x_0\in \bb R^n$ with $\Sigma(x_0)< \infty$ and, for ease of notation, set $\Sigma = \Sigma(x_0)$. For $0< r< R$, let $A(r, R) = B_R(x_0)\setminus \bar B_r(x_0)$ denote the open annulus centered at $x_0$ with inner radius $r$ and outer radius $R$. For any $\epsilon\in (0, \Sigma]$, any $R\in (\Sigma + \epsilon, \infty)$ and any $\sigma\in [\Sigma, \Sigma + \epsilon]$ we have 
\begin{equation*}
	\mc P(x_0, \sigma)
	\subset A(\Sigma, \Sigma + \epsilon) \cup A(R, \infty)\cup\left(\bar A(\Sigma + \epsilon, R)\cap \mc P(x_0, \sigma)\right)
\end{equation*}
and thus
\begin{equation*}
	\mc P(x_0, \sigma)^{x_0, \sigma}
	\subset A(\Sigma, \Sigma + \epsilon)^{x_0, \sigma} 
	\cup A(R, \infty)^{x_0, \sigma}
	\cup \left(\bar A(\Sigma + \epsilon, R)\cap \mc P(x_0, \sigma)\right)^{x_0, \sigma}. 
\end{equation*}
Moreover we have both $|A(\Sigma, \Sigma + \epsilon)^{x_0, \sigma}|\leq C\epsilon$ and $|A(R, \infty)^{x_0, \sigma}|\leq C R^{-n}$ for some $\sigma$-independent positive constant $C$. 
Using the change of variable $y= x^{x_0, \sigma}$ and since $|x - x_0|> \sigma$ whenever $\sigma\leq \Sigma + \epsilon$ and $x\in A(\Sigma + \epsilon, R)$ we have
\begin{equation*}
\begin{split}
	|\left(\bar A(\Sigma + \epsilon, R)\cap \mc P(x_0, \sigma)\right)^{x_0,\sigma}|
	& = \int_{\left(\bar A(\Sigma + \epsilon, R)\cap \mc P(x_0, \sigma)\right)^{x_0,\sigma}}\; \d y\\
	& = \int_{\bar A(\Sigma + \epsilon, R)\cap \mc P(x_0, \sigma)}\left(\frac{\sigma}{|x - x_0|}\right)^{2n}\; \d x\\
	& \leq |\bar A(\Sigma + \epsilon, R)\cap \mc P(x_0, \sigma)|. 
\end{split}
\end{equation*}
Since $\epsilon>0$ is small and arbitrary and $R\gg 1$ is large and arbitrary, to complete the proof it suffices to show that the following limit holds for every $\epsilon\in (0, \Sigma]$ and all $R\in (\Sigma, \infty)$: 
\begin{equation*}
	\lim_{\sigma\to \Sigma(x_0)^+}|\bar A(\Sigma(x_0) + \epsilon, R)\cap \mc P(x_0, \sigma)| = 0. 
\end{equation*}
Proceeding by way of contradiction, suppose $\epsilon\in (0, \Sigma]$, $R\in (\Sigma + \epsilon, \infty)$, $\ell>0$ and $(\sigma_k)_{k = 1}^\infty$ is a sequence for which both $\sigma_k\to \Sigma^+$ and 
\begin{equation*}
	\ell
	< |\bar A(\Sigma(x_0) + \epsilon, R)\cap \mc P(x_0, \sigma_k)|
	\quad \text{ for all }k. 
\end{equation*}
The assumption $u> u_{x_0, \Sigma}$ in $\bb R^n\setminus \bar B_{\Sigma}(x_0)$ guarantees the existence of $\delta>0$ such that $u - u_{x_0, \Sigma(x_0)}\geq \delta$ in $\bar A(\Sigma + \epsilon, R)$. For any $y\in \bar A(\Sigma+ \epsilon, R)\cap \mc P(x_0, \sigma_k)$ and for $k$ sufficiently large we have
\begin{equation*}
\begin{split}
	\delta
	& \leq u(y) - u_{x_0, \Sigma}(y)\\
	& < u_{x_0, \sigma_k}(y) - u_{x_0, \Sigma}(y)\\
	& = u(y^{x_0, \sigma_k}) - u(y^{x_0, \Sigma}) + 2\log \frac{\sigma_k}{\Sigma}\\
	& \leq u(y^{x_0, \sigma_k}) - u(y^{x_0, \Sigma}) + \frac \delta 2. 
\end{split}
\end{equation*}
Therefore, still for $k$ large, we have 
\begin{equation}
\label{eq:large_measure_for_contradiction}
\begin{split}
	\ell
	& < |\bar A(\Sigma + \epsilon, R)\cap \mc P(x_0, \sigma_k)|\\
	& \leq |\{y\in \bar A(\Sigma + \epsilon, R): \frac\delta 2 \leq u(y^{x_0, \sigma_k}) - u(y^{x_0, \Sigma})\}|. 
\end{split}
\end{equation}
On the other hand, by the uniform continuity of $u$ on $\bar A(\Sigma^2/ R, \Sigma + \epsilon)$ we have
\begin{equation*}
	\lim_{k\to\infty}|\{y\in \bar A(\Sigma + \epsilon, R): \frac \delta 2\leq u(y^{x_0, \sigma_k}) - u(y^{x_0, \Sigma(x_0)})\}|
	= 0, 
\end{equation*}
which contradicts estimate \eqref{eq:large_measure_for_contradiction}. 
\end{proof}
\begin{lemma}
\label{lemma:critical_radius_infinite}
If $u$ is a distributional solution of \eqref{eq:entire_nonlocal} that satisfies both $\frac{n - \mu}{2n - \mu}< \frac V{|\bb S^n|}< 1$ and \eqref{eq:asymtpotic_growth_assumption} then $\Sigma(x_0) = +\infty$ for every $x_0\in \bb R^n$. 
\end{lemma}
\begin{proof}
Proceeding by way of contradiction, suppose $x_0\in \bb R^n$ with $\Sigma(x_0)< \infty$. Let $(\sigma_k)_{k = 1}^\infty\subset (\Sigma(x_0), \infty)$ satisfy $\lim_k\sigma_k = \Sigma(x_0)$. In view of the definition of $\Sigma(x_0)$ (see \eqref{eq:symmetry_radius}), for every $k\in \bb N$ there is $r_k\in (\Sigma(x_0), \sigma_k]$ and $x_k\in \bb R^n\setminus \overline B_{r_k}(x_0)$ for which either 
\begin{equation*}
	I_\mu[e^{\varrho u}](x_k)
	< I_\mu[e^{\varrho u_{x_0, r_k}}](x_k). 
\end{equation*}
or
\begin{equation*}
	u(x_k)< u_{x_0, r_k}(x_k)
\end{equation*}
By passing to a suitable subsequence we may assume that one of these inequalities holds for every $k$. 
\begin{enumerate}[label = {\bf Case \arabic*.}, ref= {\bf Case \arabic*}, wide = 0pt]
	\item \label{case:Imu_violates} $I_\mu[e^{\varrho u}](x_k)< I_\mu[e^{\varrho u_{x_0, r_k}}](x_k)$ for all $k$. \\
	The \ref{case:Imu_violates} assumption guarantees that $x_k\in \mc Q(x_0, r_k)$ for all $k$, so Lemma \ref{lemma:Qnonempty} guarantees that 
	\begin{equation}
	\label{eq:P_always_positive_measure}
		|\mc P(x_0, r_k)|> 0
		\quad \text{ for all }k. 
	\end{equation}
	On the other hand, Lemma \ref{lemma:bad_set_vanishing_measure} guarantees that $\lim_k|\mc P(x_0, r_k)^{x_0, r_k}| = 0$, so we have
	\begin{equation*}
		\lim_{k\to\infty}\|1 + |\log|\cdot - x_0||\|_{L^{n/\varrho}(\mc P(x_0, r_k)^{x_0, r_k})} = 0. 
	\end{equation*}
	This equality, together with Lemma \ref{lemma:MS_start_main_estimate} (applied with $\Sigma = 1 + \Sigma(x_0)$) implies that 
	\begin{equation*}
		\|e^{\varrho u_{x_0, r_k}}- e^{\varrho u}\|_{L^{n/\varrho}(\mc P(x_0, r_k))}
		\leq \frac 12\|e^{\varrho u_{x_0, r_k}}- e^{\varrho u}\|_{L^{n/\varrho}(\mc P(x_0, r_k))}
	\end{equation*}
	whenever $k$ is sufficiently large. For any such $k$ we obtain $|\mc P(x_0, r_k)| = 0$ thereby contradicting inequality \eqref{eq:P_always_positive_measure}. 
	\item $u(x_k)< u_{x_0, r_k}(x_k)$ for all $k$. \\
	In this case, estimating as in \eqref{eq:get_H1_H2} (with $x = x_k$ and $\sigma = r_k$) we obtain 
	\begin{equation*}
	\begin{split}
		0
		< & \; \frac{e^{\varrho u_{x_0, r_k}(x_k)} - e^{\varrho u(x_k)}}{\varrho c_n(n - 1)!e^{\varrho u_{x_0, r_k}(x_k)}}\\
		\leq & \; \int_{\mc Q(x_0, r_k)}\mc L(x_0, r_k; x_k, y)\left(I_\mu[e^{\varrho u_{x_0, r_k}}](y) - I_\mu[e^{\varrho u}](y)\right)e^{\varrho u_{x_0, r_k}(y)}\; \d y\\
		& + \int_{\mc P(x_0, r_k)}\mc L(x_0, r_k; x_k, y)I_\mu[e^{\varrho u}](y)\left(e^{\varrho u_{x_0, r_k}(y)}- e^{\varrho u(y)}\right)\; \d y, 
	\end{split}
	\end{equation*}
	which implies that $|\mc Q(x_0, r_k)| + |\mc P(x_0, r_k)|> 0$ for all $k$. An application of Lemma \ref{lemma:Qnonempty} now guarantees that $|\mc P(x_0, r_k)|>0$ for all $k$, so the argument of \ref{case:Imu_violates} (starting from estimate \eqref{eq:P_always_positive_measure}) yields a contradiction. 
\end{enumerate}
\end{proof}
\begin{proof}[Proof of Lemma \ref{lemma:rule_out_slow_decay}]
Proceeding by way of contradiction, suppose $\frac{n - \mu}{2n - \mu}< \frac V{|\bb S^n|}< 1$. Lemma \ref{lemma:critical_radius_infinite} guarantees that for every $(x_0, \sigma)\in \bb R^n\times(0, \infty)$ there holds
\begin{equation*}
	\left(\frac\sigma{|x -x_0|}\right)^2 e^{u(x^{x_0, \sigma})}
	\leq e^{u(x)}
	\quad \text{ for all }x\in \bb R^n\setminus B_\sigma(x_0). 
\end{equation*}
Applying Lemma \ref{lemma:implies_const_or_infty} of the appendix with $\gamma = 2$ and $f = e^u$ implies that either $e^u\equiv \text{const}$ or $e^u\equiv +\infty$, both of which contradict the assumption $e^u\in L^n(\bb R^n)$. 
\end{proof}
\subsection{Ruling out Fast Decay}
\label{ss:rule_out_fast_decay}
The following lemma is the main result of this subsection. 
\begin{lemma}
\label{lemma:rule_out_fast_decay}
If $u$ is a distributional solution to \eqref{eq:entire_nonlocal} for which \eqref{eq:V_lower_bound_assumption} and \eqref{eq:asymtpotic_growth_assumption} are satisfied, then $V\leq |\bb S^n|$.  
\end{lemma}
Similarly to the proof of Lemma \ref{lemma:rule_out_slow_decay} given in Subsection \ref{ss:rule_out_slow_decay}, the proof of Lemma \ref{lemma:rule_out_fast_decay} proceeds by way of contradiction and uses the method of moving spheres to obtain a contradiction. However, in this setting, the method of moving spheres initiates with spheres having large radii and then shrinks the spheres until the radii are zero.
For $(x_0, \sigma)\in \bb R^n\times (0, \infty)$ define the sets
\begin{equation*}
\begin{split}
	\tilde {\mc P}(x_0, \sigma)
	& = \{x\in \bb R^n\setminus \overline B_\sigma(x_0): u_{x_0, \sigma}(x)< u(x)\}\\
	\tilde{\mc Q}(x_0, \sigma)
	& = \{x\in \bb R^n\setminus \overline B_\sigma(x_0): I_\mu[e^{\varrho u_{x_0, \sigma}}](x)< I_\mu[e^{\varrho u}](x)\}. 
\end{split}
\end{equation*}
\begin{lemma}
\label{lemma:one_implies_both}
Let $u$ be a distributional solution to \eqref{eq:entire_nonlocal} satisfying both \eqref{eq:V_lower_bound_assumption} and \eqref{eq:asymtpotic_growth_assumption}. If $(x_0, \sigma)\in \bb R^n\times(0, \infty)$ and if $\tilde{\mc Q}(x_0, \sigma)\neq\emptyset$ then $|\tilde{\mc P}(x_0, \sigma)|> 0$. In particular, if $(x_0, \sigma)\in \bb R^n\times(0, \infty)$ and if $|\tilde{\mc P}(x_0, \sigma)|+ |\tilde{\mc Q}(x_0, \sigma)|> 0$, then $|\tilde{\mc P}(x_0, \sigma)|> 0$. 
\end{lemma}
\begin{proof}
For every $x\in \tilde{\mc Q}(x_0, \sigma)$, using Lemma \ref{lemma:Imu_difference} and the fact that $\mc K(x_0, \sigma; x, y)$ as defined in \eqref{eq:Imu_difference_integral_kernel} satisfies $\mc K(x_0, \sigma; x, y)\leq |x- y|^{-\mu}$ we have
\begin{equation*}
\begin{split}
	0
	< & \; I_\mu[e^{\varrho u}](x) - I_\mu[e^{\varrho u_{x_0,  \sigma}}](x)\\
	& 
	\; \int_{\tilde{\mc P}(x_0, \sigma)}\frac 1{|x- y|^\mu}\left(e^{\varrho u(y)} - e^{\varrho u_{x_0, \sigma}(y)}\right)\; \d y\\
	= & \; I_\mu[(e^{\varrho u} - e^{\varrho u_{x_0, \sigma}})\chi_{\tilde{\mc P}(x_0, \sigma)}](x), 
\end{split}
\end{equation*}
from which we deduce that $|\tilde{\mc P}(x_0, \sigma)|> 0$. 
\end{proof}
\begin{lemma}
\label{lemma:MS_start_main_estimate_2}
Let $u$ be a distributional solution to \eqref{eq:entire_nonlocal} for which \eqref{eq:asymtpotic_growth_assumption} holds. If $V> |\bb S^n|$ then for every $\Sigma>0$ there is a positive constant $C$ depending on $\Sigma$, $\mu$, $\|e^u\|_{L^n(\bb R^n)}$, $\|e^u\|_{L^\infty(\bb R^n)}$ and $\|I_\mu[e^{\varrho u}]\|_{L^\infty(\bb R^n)}$ such that for every $(x_0, \sigma)\in \bb R^n\times[\Sigma, \infty)$ there holds
\begin{equation}
\label{eq:see_you_again}
\begin{split}
	\lefteqn{\|e^{\varrho u} - e^{\varrho u_{x_0, \sigma}}\|_{L^{n/\varrho}(\tilde{\mc P}(x_0, \sigma))}}\\
	& \leq C\|e^{\varrho u}(1 + |\log|\cdot - x_0||)\|_{L^{n/\varrho}(\tilde{\mc P}(x_0,\sigma))} \|e^{\varrho u} - e^{\varrho u_{x_0, \sigma}}\|_{L^{n/\varrho}(\tilde{\mc P}(x_0, \sigma))}. 
\end{split}
\end{equation}
We emphasize that $C$ is independent of $\sigma\in [\Sigma, \infty)$. 
\end{lemma}
\begin{proof}
The proof is similar to the proof of Lemma \ref{lemma:MS_start_main_estimate}, so only an outline of the proof is provided. Fix $\Sigma>0$ and $(x_0, \sigma)\in \bb R^n\times[\Sigma, \infty)$. For $x\in \tilde{\mc P}(x_0, \sigma)$ set
\begin{equation*}
\begin{split}
	\tilde H_1(x_0, \sigma; x)
	& = \int_{\tilde{\mc Q}(x_0, \sigma)}\bar {\mc L}(x_0, \sigma; x,y)I_\mu[(e^{\varrho u} - e^{\varrho u_{x_0, \sigma}})\chi_{\tilde{\mc P}(x_0, \sigma)}](y)e^{\varrho u(y)}\; \d y\\
	\tilde H_2(x_0, \sigma; x)
	& = \int_{\tilde{\mc P}(x_0, \sigma)}\bar {\mc L}(x_0, \sigma; x,y)I_\mu[e^{\varrho u_{x_0, \sigma}}](y)\left(e^{\varrho u(y)} - e^{\varrho u_{x_0, \sigma}(y)}\right)\; \d y. 
\end{split}
\end{equation*}
where $\bar {\mc L}(x_0, \sigma; x,y)$ is as in \eqref{eq:L_bar}. For any $x\in \tilde{\mc P}(x_0, \sigma)$, estimating similarly to \eqref{eq:get_H1_H2} (using the Mean Value Theorem, Lemma \ref{lemma:u_difference} and Lemma \ref{lemma:Imu_difference}) we have
\begin{equation*}
\begin{split}
	0
	< & \; \frac{e^{\varrho u(x)} - e^{\varrho u_{x_0, \sigma}(x)}}{c_n(n - 1)!\varrho e^{\varrho u(x)}}\\
	\leq & \; \frac{u(x) - u_{x_0, \sigma}(x)}{c_n(n - 1)!}\\
	= &\; (V- |\bb S^n|)\log\frac\sigma{|x- x_0|}\\
	& + \int_{\bb R^n\setminus B_\sigma(x_0)}\mc L(x_0, \sigma; x, y)\left(I_\mu[e^{\varrho u}](y)e^{\varrho u(y)} - I_\mu[e^{\varrho u_{x_0, \sigma}}](y)e^{\varrho u_{x_0, \sigma}(y)}\right)\; \d y\\
	\leq & \; \int_{\tilde{\mc Q}(x_0, \sigma)}\mc L(x_0, \sigma; x, y)\left(I_\mu[e^{\varrho u}](y) - I_\mu[e^{\varrho u_{x_0, \sigma}}](y)\right)e^{\varrho u(y)}\; \d y\\
	& + \int_{\tilde{\mc P}(x_0, \sigma)}I_\mu[e^{\varrho u_{x_0,\sigma}}](y)\left(e^{\varrho u(y)} - e^{\varrho u_{x_0, \sigma}(y)}\right)\; \d y\\
	& \leq \tilde H_1(x_0, \sigma; x) + \tilde H_2(x_0, \sigma; x). 
\end{split}
\end{equation*}
Consequently, 
\begin{equation}
\label{eq:belong}
	\|e^{\varrho u} - e^{\varrho u_{x_0, \sigma}}\|_{L^{n/\varrho}(\tilde{\mc P}(x_0, \sigma))}
	\leq C\sum_{i = 1}^2\|e^{\varrho u}\tilde H_i(x_0, \sigma; \cdot)\|_{L^{n/\varrho}(\tilde{\mc P}(x_0, \sigma))}
\end{equation}
for some constant $C>0$ that depends only on $n$. Estimating similarly to the pointwise estimates of $H_i(x_0, \sigma; x)$ in the proof of Lemma \ref{lemma:MS_start_main_estimate} we obtain $C>0$ independent of $\sigma\in [\Sigma, \infty)$ such that for each $i\in \{1, 2\}$ and every $x\in \tilde{\mc P}(x_0, \sigma)$ there holds 
\begin{equation}
\label{eq:FJM}
\begin{split}
	0
	& \leq \tilde H_i(x_0, \sigma; x)\\
	& \leq C\left(1 + \log\frac{|x- x_0|}{\Sigma}\right)\|e^{\varrho u} - e^{\varrho u_{x_0, \sigma}}\|_{L^{n/\varrho}(\tilde{\mc P}(x_0, \sigma))}\\
	& \leq C\left(1 + |\log|x- x_0||\right)\|e^{\varrho u} - e^{\varrho u_{x_0, \sigma}}\|_{L^{n/\varrho}(\tilde{\mc P}(x_0, \sigma))}. 
\end{split}
\end{equation}
Bringing estimate \eqref{eq:FJM} back to \eqref{eq:belong} yields the asserted estimate. 
\end{proof}
\begin{lemma}
\label{lemma:own_way}
Let $u$ be a distributional solution to \eqref{eq:entire_nonlocal} for which $V> |\bb S^n|$ and for which \eqref{eq:asymtpotic_growth_assumption} holds. For any $x_0\in \bb R^n$ there is $\tilde \sigma(x_0)>1$ such that both 
\begin{equation}
\label{eq:inverted_dominates}
	u\leq u_{x_0, \sigma} \quad \text{ in }\bb R^n\setminus B_\sigma(x_0)
\end{equation}
and
\begin{equation}
\label{eq:Imu_inverted_dominates}
	I_\mu[e^{\varrho u}]\leq I_\mu[e^{\varrho u_{x_0, \sigma}}]
	\quad \text{ in }\bb R^n\setminus B_\sigma(x_0)
\end{equation}
whenever $\sigma\in (\tilde \sigma(x_0), \infty)$. 
\end{lemma}
\begin{proof}
Let $x_0\in \bb R^n$ and apply Lemma \ref{lemma:MS_start_main_estimate_2} with $\Sigma = 1$ to obtain a positive constant $C$ for which \eqref{eq:see_you_again} holds whenever $\sigma\geq 1$. Theorem \ref{theorem:asymptotic} and assumption \eqref{eq:asymtpotic_growth_assumption} imply the existence of $b\in \bb R$ for which 
\begin{equation}
\label{eq:way_back}
	u(x)
	= b - \frac{2V}{|\bb S^n|}\log|x| + \circ(\log|x|)
\end{equation}
as $|x|\to\infty$. Since $V> |\bb S^n|> |\bb S^n|/2$ we fix $\epsilon>0$ such that $\frac{2V}{|\bb S^n|} - \epsilon \geq 1$. For any such $\epsilon$ \eqref{eq:way_back} guarantees the existence of $R = R(\epsilon)\gg 1 + |x_0|$ such that 
\begin{equation*}
	e^{nu(x)}
	\leq e^{nb}|x|^{-n\left(\frac{2V}{|\bb S^n|} - \epsilon\right)}
	\quad \text{ for }x\in \bb R^n\setminus B_R(x_0). 
\end{equation*}
For any such $R$ we have 
\begin{equation}
\label{eq:integrable_at_infinity}
	e^{\varrho u}(1 + |\log|\cdot - x_0||)\in L^{n/\varrho}(\bb R^n\setminus B_R(x_0)), 
\end{equation}
so 
\begin{equation*}
	\|e^{\varrho u}(1 + |\log|\cdot - x_0||)\|_{L^{n/\varrho}(\tilde{\mc P}(x_0, \sigma))}
	\leq \|e^{\varrho u}(1 + |\log|\cdot - x_0||)\|_{L^{n/\varrho}(\bb R^n\setminus B_\sigma(x_0))}
	\to 0
\end{equation*}
as $\sigma\to\infty$. Therefore, for $\tilde \sigma(x_0)$ sufficiently large and for $\sigma\geq \tilde \sigma(x_0)$ estimate \eqref{eq:see_you_again} gives
\begin{equation*}
	\|e^{\varrho u}-e^{\varrho u_{x_0, \sigma}}\|_{L^{n/\varrho}(\tilde{\mc P}(x_0, \sigma))}
	\leq \frac 12 \|e^{\varrho u}-e^{\varrho u_{x_0, \sigma}}\|_{L^{n/\varrho}(\tilde{\mc P}(x_0, \sigma))}
\end{equation*}
and hence $|\tilde{\mc P}(x_0, \sigma)| = 0$. For any such $\sigma$ Lemma \ref{lemma:one_implies_both} implies that $|\tilde{\mc Q}(x_0, \sigma)|= 0$. 
\end{proof}
For any $x_0\in \bb R^n$, Lemma \ref{lemma:own_way} guarantees that
\begin{equation}
\label{eq:Sigma_tilde_x0}
	\tilde \Sigma(x_0)
	= \inf\{\beta> 0: \text{ both \eqref{eq:inverted_dominates} and \eqref{eq:Imu_inverted_dominates} hold for all $\sigma> \beta$}\}
\end{equation}
is well-defined in $[0, \infty)$. 
\begin{lemma}
\label{lemma:tilde_critical_sphere_all_good}
Let $u$ be a distributional solution to \eqref{eq:entire_nonlocal} for which $V > |\bb S^n|$ and \eqref{eq:asymtpotic_growth_assumption} holds. If $x_0\in \bb R^n$ with $\tilde\Sigma(x_0)> 0$ then $|\tilde{\mc P}(x_0, \tilde\Sigma(x_0))| + |\tilde{\mc Q}(x_0, \tilde\Sigma(x_0))| = 0$. 
\end{lemma}
\begin{proof}
Proceeding by way of contradiction, suppose $x_0\in \bb R^n$ with $\tilde \sigma(x_0)> 0$ and $|\tilde{\mc P}(x_0, \sigma)| + |\tilde{\mc Q}(x_0, \sigma)|> 0$. For ease of notation we write $\tilde \Sigma$ in place of $\tilde \Sigma(x_0)$. Choose $\delta>0$ and $R\gg 1$ such that
\begin{equation*}
	3|\tilde{\mc P}(x_0, \tilde \Sigma)|
	\leq 4|\{x\in A(\tilde \Sigma, R): \delta< u(x) - u_{x_0,  \tilde \Sigma}(x)\}|. 
\end{equation*}
If $x\in A(\tilde \Sigma, R)$ with $\delta< u(x)- u_{x_0, \tilde \Sigma}(x)$ then for any $\sigma\in (\tilde \Sigma, 2\tilde\Sigma)$ sufficiently close to $\tilde \Sigma$ 
we have 
\begin{equation*}
\begin{split}
	0
	& \geq u(x) - u_{x_0, \sigma}(x)\\
	& \geq \delta + u_{x_0, \tilde \Sigma}(x) - u_{x_0, \sigma}(x)\\
	& = \delta + u(x^{x_0, \tilde \Sigma}) - u(x^{x_0, \sigma}) + 2\log\frac{\tilde \Sigma}{\sigma}\\
	& > \frac\delta 2 + u(x^{x_0, \tilde \Sigma}) - u(x^{x_0, \sigma}). 
\end{split}
\end{equation*}
Therefore, 
\begin{equation}
\label{eq:it_goes_on}
	3|\tilde{\mc P}(x_0, \tilde \Sigma)|
	\leq 4|\{x\in A(\tilde \Sigma, R): \frac\delta2< u(x^{x_0, \sigma}) - u(x^{x_0, \tilde \Sigma})\}|. 
\end{equation}
On the other hand, since $u$ in uniformly continuous on $A(R^{-1}\tilde \Sigma^2, 4\tilde \Sigma^2)$, by choosing $\sigma\in (\tilde \Sigma, 2\tilde \Sigma)$ closer to $\tilde \Sigma$ if necessary we have 
\begin{equation*}
	|\{x\in A(\tilde \Sigma, R): \frac\delta2< u(x^{x_0, \sigma}) - u(x^{x_0, \tilde \Sigma})\}|
	< \frac 12|\tilde{\mc P}(x_0, \tilde \Sigma)|, 
\end{equation*}
which contradicts inequality \eqref{eq:it_goes_on}. 
\end{proof}
\begin{lemma}
\label{lemma:into_the_night}
Let $u$ be a distributional solution to \eqref{eq:entire_nonlocal} for which $V > |\bb S^n|$ and \eqref{eq:asymtpotic_growth_assumption} holds. If $x_0\in \bb R^n$ with $\tilde \Sigma(x_0)> 0$ then 
\begin{equation*}
	u(x)< u_{x_0, \tilde\Sigma(x_0)}(x)
	\quad \text{ for all }x\in \bb R^n\setminus\overline B_{\tilde\Sigma(x_0)}(x_0)
\end{equation*}
and
\begin{equation*}
	I_\mu[e^{\varrho u}](x)< I_\mu[e^{\varrho u_{x_0, \tilde\Sigma(x_0)}}](x)
	\quad \text{ for all }x\in \bb R^n\setminus\overline B_{\tilde\Sigma(x_0)}(x_0)
\end{equation*}
\end{lemma}
\begin{proof}
Let $x_0\in \bb R^n$ with $\tilde \Sigma(x_0)> 0$ and, for ease of notation, write $\tilde \Sigma$ in place of $\tilde \Sigma(x_0)$. Lemma \ref{lemma:tilde_critical_sphere_all_good} guarantees that both $u\leq u_{x_0, \tilde \Sigma}$ and $I_\mu[e^{\varrho u}]\leq I_\mu[e^{\varrho u_{x_0, \tilde \Sigma}}]$ in $\bb R^n\setminus B_{\tilde \Sigma}(x_0)$. Therefore, for any $x\in \bb R^n\setminus \bar B_{\tilde \Sigma}(x_0)$ Lemma \ref{lemma:u_difference} gives
\begin{equation*}
\begin{split}
	u_{x_0, \tilde \Sigma}(x)&  - u(x)\\
	& > u_{x_0, \tilde \Sigma}(x) - u(x) + c_n(n - 1)!(V- |\bb S^n|)\log\frac{\tilde \Sigma}{|x - x_0|}\\
	& = \int_{\bb R^n\setminus B_{\tilde \Sigma}(x_0)}\mc L(x_0, \tilde \Sigma; x, y)\left(I_\mu[e^{\varrho u_{x_0, \tilde \Sigma}}(y)e^{\varrho u_{x_0, \tilde\Sigma}(y)} - I_\mu[e^{\varrho u}](y)e^{\varrho u(y)}\right)\; \d y\\
	& \geq 0, 
\end{split}
\end{equation*}
thereby establishing the first of the asserted equalities. The second of the asserted equalities now follows immediately from Lemma \ref{lemma:Imu_difference}. 
\end{proof}
\begin{lemma}
\label{lemma:tilde_approach_critial_radius}
If $x_0\in \bb R^n$ with $\tilde \Sigma(x_0)>0$ then 
\begin{equation*}
	\lim_{\sigma\to \tilde \Sigma(x_0)^-}\|e^{\varrho u}(1 + |\log|\cdot - x_0||)\|_{L^{n/\varrho}(\tilde{\mc P}(x_0, \sigma))}
	= 0. 
\end{equation*}
\end{lemma}
\begin{proof}
Let $x_0\in \bb R^n$ with $\tilde\Sigma(x_0)> 0$ and for ease of notation write $\tilde \Sigma$ in place of $\tilde \Sigma(x_0)$. As in the proof of Lemma \ref{lemma:own_way} (see \eqref{eq:integrable_at_infinity}) we have $e^{\varrho u}(1 + |\log|\cdot - x_0||)\in L^{n/\varrho}(\bb R^n\setminus B_R)$ whenever $R$ is sufficiently large. Using this together with Proposition \ref{prop:distributional_solutions_smooth} one can verify that $e^{\varrho u}(1 + |\log|\cdot - x_0||)\in L^{n/\varrho}(\bb R^n)$. Now, for any $\epsilon\in (0, \tilde \Sigma/2]$ and any $\sigma\in [\tilde \Sigma - \epsilon, \tilde \Sigma]$ we have
\begin{equation*}
	\tilde{\mc P}(x_0, \sigma)
	\subset A(\tilde \Sigma - \epsilon, \tilde \Sigma)\cup A(R, \infty)\cup\left(\tilde {\mc P}(x_0, \sigma)\cap A(\tilde \Sigma - \epsilon, R)\right). 
\end{equation*}
Moreover we have $|A(\tilde \Sigma - \epsilon, \Tilde \sigma)|\leq C\epsilon$ for some $\sigma$-independent positive constant $C$ and we have 
\begin{equation*}
	\lim_{R\to 0}\|e^{\varrho u}(1 + |\log|\cdot - x_0||)\|_{L^{n/\varrho}(A(R, \infty))}
	= 0. 
\end{equation*}
Therefore, to complete the proof it suffices to show that the following limit holds for every $\epsilon\in (0, \tilde \Sigma/2]$ and every $R\gg 1$: 
\begin{equation*}
	\lim_{\sigma \to \Tilde \Sigma^-}|\tilde{\mc P}(x_0, \sigma)\cap A(\tilde \Sigma - \epsilon, R)| = 0. 
\end{equation*}
We do so by way of contradiction. Accordingly, suppose $\epsilon\in (0, \tilde \Sigma/2]$, $R\gg 1$, $\ell> 0$ and $(\sigma_k)_{k = 1}^\infty\subset (\tilde \Sigma/2, \tilde \Sigma)$ satisfy both $\sigma_k\to \tilde \Sigma^-$ and
\begin{equation*}
	\ell
	< |\tilde{\mc P}(x_0, \sigma_k)\cap A(\tilde \Sigma - \epsilon, R)|
	\quad \text{ for all $k\in \bb N$}. 
\end{equation*}
In view of Lemma \ref{lemma:into_the_night} there is $\delta> 0$ such that 
\begin{equation*}
	u_{x_0, \tilde \Sigma}(x) - u(x)\geq \delta
	\quad \text{ in }A(\tilde \Sigma - \epsilon, R). 
\end{equation*}
For any $y\in A(\tilde \Sigma - \epsilon, R)\cap \tilde {\mc P}(x_0, \sigma_k)$ and for sufficiently large $k$ we have
\begin{equation*}
\begin{split}
	\delta 
	& \leq u_{x_0, \tilde \Sigma}(y) - u(y)\\
	& < u_{x_0, \tilde \Sigma}(y) - u_{x_0, \sigma_k}(y)\\
	& = u(y^{x_0, \tilde \Sigma}) - u(y^{x_0, \sigma_k}) + 2\log \frac{\tilde \Sigma}{\sigma_k}\\
	& < u(y^{x_0, \tilde \Sigma}) - u(y^{x_0, \sigma_k}) + \frac \delta 2. 
\end{split}
\end{equation*}
Therefore, still for $k$ large, we have
\begin{equation}
\label{eq:everyone_has_podcast}
\begin{split}
	\ell
	& < |\tilde{\mc P}(x_0, \sigma_k)\cap A(\tilde \Sigma -\epsilon, R)|\\
	& \leq |\{y\in \bar A(\tilde \Sigma - \epsilon, R): \frac \delta 2<  u(y^{x_0, \tilde \Sigma}) - u(y^{x_0, \sigma_k})\}|. 
\end{split}
\end{equation}
On the other hand, by the uniform continuity of $u$ on $A((4R)^{-1}\tilde \Sigma^2, 2\tilde \Sigma)$ we have
\begin{equation*}
	\lim_{k\to\infty}|\{y\in \bar A(\tilde \Sigma - \epsilon, R): \frac \delta 2<  u(y^{x_0, \tilde \Sigma}) - u(y^{x_0, \sigma_k})\}|
	= 0, 
\end{equation*}
which contradicts \eqref{eq:everyone_has_podcast}. 
\end{proof}
\begin{lemma}
\label{lemma:radius_fully_shrinks}
If $u$ is a distributional solution to \eqref{eq:entire_nonlocal} for which $V > |\bb S^n|$ and \eqref{eq:asymtpotic_growth_assumption} holds then for every $x_0\in \bb R^n$ there holds $\tilde \Sigma(x_0) = 0$. 
\end{lemma}
\begin{proof}
Proceeding by way of contradiction, suppose $x_0\in \bb R^n$ with $\tilde \Sigma(x_0)> 0$ and for ease of notation write $\tilde \Sigma$ in place of $\tilde \Sigma(x_0)$. Let $(\sigma_k)_{k = 1}^\infty\subset(0, \tilde \Sigma)$ with $\lim_k\sigma_k = \tilde \Sigma$. By the definition of $\tilde \Sigma$, for every $k\in \bb N$ there is $r_k\in[\sigma_k, \tilde \Sigma)$ and $x_k\in \bb R^n\setminus \bar B_{r_k}(x_0)$ for which either 
\begin{equation*}
	I_\mu[e^{\varrho u}](x_k)> I_\mu[e^{\varrho u_{x_0, r_k}}](x_k)
\end{equation*}
or 
\begin{equation*}
	u(x_k)> u_{x_0, r_k}(x_k). 
\end{equation*}
After passing to a suitable subsequence we may assume that one of these inequalities occurs for every $k$. 
\begin{enumerate}[label = {\bf Case \arabic*.}, ref = {\bf Case \arabic*}, wide = 0pt]
	\item \label{item:spine} $I_\mu[e^{\varrho u}](x_k)> I_\mu[e^{\varrho u_{x_0, r_k}}](x_k)$ for all $k$. \\
	The \ref{item:spine} assumption guarantees that $x_k\in \tilde {\mc Q}(x_0, r_k)$ for all $k$ so Lemma \ref{lemma:one_implies_both} implies
	\begin{equation}
	\label{eq:down_my_spine}
		|\tilde{\mc P}(x_0, r_k)|> 0
		\quad \text{ for all }k. 
	\end{equation}
	On the other hand, Lemma \ref{lemma:tilde_approach_critial_radius} implies 
	\begin{equation*}
		\lim_k\|e^{\varrho u}(1 + |\log|\cdot - x_0||)\|_{L^{n/\varrho}(\tilde{\mc P}(x_0, r_k))} 
		= 0. 
	\end{equation*}
	This equality, together with Lemma \ref{lemma:MS_start_main_estimate_2} (applied with $\Sigma = \tilde \Sigma/2$) gives
	\begin{equation*}
		\|e^{\varrho u} - e^{\varrho u_{x_0, r_k}}\|_{L^{n/\varrho}(\tilde{\mc P}(x_0, r_k))}
		\leq \frac 12\|e^{\varrho u} - e^{\varrho u_{x_0, r_k}}\|_{L^{n/\varrho}(\tilde{\mc P}(x_0, r_k))}
	\end{equation*}
	whenever $k$ is sufficiently large. For any such $k$ we have $|\tilde{\mc P}(x_0, r_k)| = 0$ thereby contradicting \eqref{eq:down_my_spine}. 
	\item $u(x_k)> u_{x_0, r_k}(x_k)$ for all $k$. \\
	In this case the Mean Value Theorem and Lemma \ref{lemma:u_difference} give
	\begin{equation*}
	\begin{split}
		0
		< & \; \frac{e^{\varrho u(x_k)} - e^{\varrho u_{x_0, r_k}(x_k)}}{c_n(n - 1)!\varrho e^{\varrho u(x_k)}}\\
		\leq& \; \frac{u(x_k) - u_{x_0, r_k}(x_k)}{c_n(n- 1)!}\\
		= & \; \int_{\bb R^n\setminus B_{r_k}(x_0)}\mc L(x_0, r_k; x, y)\left(I_\mu[e^{\varrho u}](y)e^{\varrho u(y)} - I_\mu[e^{\varrho u_{x_0, r_k}}](y)e^{\varrho u_{x_0, r_k}(y)}\right)\; \d y\\
		& + (V - |\bb S^n|)\log\frac{r_k}{|x_k - x_0|}\\
		\leq & \; \int_{\tilde{\mc Q}(x_0, r_k)}\mc L(x_0, r_k; x, y)\left(I_\mu[e^{\varrho u}](y) - I_\mu[e^{\varrho u_{x_0, r_k}}](y)\right)e^{\varrho u(y)}\; \d y\\
		& + \int_{\tilde{\mc P}(x_0, r_k)}\mc L(x_0, r_k; x, y)I_\mu[e^{\varrho u_{x_0, r_k}}](y)\left(e^{\varrho u(y)} - e^{\varrho u_{x_0, r_k}(y)}\right)\; \d y, 
	\end{split}
	\end{equation*}
	from which we deduce $|\tilde{\mc P}(x_0, r_k)| + |\tilde{\mc Q}(x_0, r_k)| > 0$ for all $k$. Lemma \ref{lemma:one_implies_both} now guarantees that $|\tilde{\mc P}(x_0, r_k)|> 0$ for all $k$, so arguing as in \ref{item:spine} (starting with \eqref{eq:down_my_spine}) yields a contradiction. 
\end{enumerate}
\end{proof}
\begin{proof}[Proof of Lemma \ref{lemma:rule_out_fast_decay}]
Proceeding by way of contradiction, suppose $u$ is a solution to \eqref{eq:entire_nonlocal} satisfying both \eqref{eq:asymtpotic_growth_assumption} and $V>|\bb S^n|$. Lemma \ref{lemma:radius_fully_shrinks} guarantees that 
\begin{equation*}
	e^{u(x)}\leq\left(\frac\sigma{|x - x_0|}\right)^2 e^{u(x^{x_0, \sigma})}
\end{equation*}
whenever $(x_0, \sigma)\in \bb R^n\times (0, \infty)$ and $x\in \bb R^n\setminus B_\sigma(x_0)$. Applying Lemma \ref{lemma:implies_const_or_infty} of the appendix with $f = e^{-u}$ and $\gamma = -2$ implies that either $e^{-u}\equiv +\infty$ or $e^{-u} \equiv \text{const}$. In the former case we have $e^{nu}\equiv 0$ and this contradicts the assumption $V> |\bb S^n|$. In the latter case we have $e^u\equiv \text{const}$ and this contradicts the assumption $e^u\in L^n(\bb R^n)$. 
\end{proof}
\section{Classification of Solutions}
\label{s:classification}
Proposition \ref{prop:precise_decay} guarantees that every solution to \eqref{eq:entire_nonlocal} that satisfies both \eqref{eq:V_lower_bound_assumption} and \eqref{eq:asymtpotic_growth_assumption} must also satisfy $V = |\bb S^n|$. In this section we use this equality and one final application of the method of moving spheres to prove Theorem \ref{theorem:classification}. Accordingly, suppose throughout this section that $u$ is a distributional solution to \eqref{eq:entire_nonlocal} that satisfies both \eqref{eq:V_lower_bound_assumption} and \eqref{eq:asymtpotic_growth_assumption}. Lemma \ref{lemma:moving_spheres_can_start} guarantees that for $x_0\in \bb R^n$, the quantity $\Sigma(x_0)$ given in \eqref{eq:symmetry_radius} is well-defined in $(0, \infty]$.  
\begin{lemma}
\label{lemma:symmetric_about_symmetry_sphere}
If $x_0\in \bb R^n$ satisfies $\Sigma(x_0)< \infty$ then both 
\begin{equation*}
	u\equiv u_{x_0, \Sigma(x_0)} \quad \text{ in }\bb R^n\setminus\{x_0\}
\end{equation*}
and 
\begin{equation*}
	I_\mu[e^{\varrho u}]\equiv I_\mu[e^{\varrho u_{x_0, \Sigma(x_0)}}]
	\quad \text{ in }\bb R^n\setminus \{x_0\}. 
\end{equation*}
\end{lemma}
\begin{proof}
Suppose $x_0\in \bb R^n$ with $\Sigma(x_0)< \infty$ and to ease the notation, set $\Sigma = \Sigma(x_0)$. In view of Lemmata \ref{lemma:Imu_difference} and \ref{lemma:u_difference} it suffices to show that both $u\equiv u_{x_0, \Sigma}$ and $I_\mu[e^{\varrho u}]\equiv I_\mu[e^{\varrho u_{x_0, \Sigma}}]$  in $\bb R^n\setminus B_{x_0, \Sigma}(x_0)$. From Lemma \ref{lemma:critical_radius_soft_inequalities} we have both 
\begin{equation}
\label{eq:both_soft_signs}
	u\geq u_{x_0, \Sigma}
	\quad \text{ and }\quad
	I_\mu[e^{\varrho u}]\geq I_\mu[e^{\varrho u_{x_0, \Sigma}}]
	\quad \text{ in }\bb R^n\setminus B_\Sigma(x_0).
\end{equation}
Next we show that  exactly one of the following holds: 
\begin{enumerate}[label = {\bf A\arabic*.}, ref = {\bf A\arabic*}]
	\item \label{item:both_always_bigger} $u(x)> u_{x_0, \Sigma}(x)$ and $I_\mu[e^{\varrho u}](x)> I_\mu[e^{\varrho u_{x_0, \Sigma}}](x)$ for all $x\in \bb R^n\setminus \overline B_{\Sigma}(x_0)$. 
	\item \label{item:both_always_same}$u(x)= u_{x_0, \Sigma}(x)$ and $I_\mu[e^{\varrho u}](x)= I_\mu[e^{\varrho u_{x_0, \Sigma}}](x)$ for all $x\in \bb R^n\setminus \overline B_{\Sigma}(x_0)$. 
\end{enumerate}
To verify this alternative, suppose alternative \ref{item:both_always_bigger} fails. We separately consider the case where $u(x) = u_{x_0, \Sigma}(x)$ for some $x\in \bb R^n\setminus \bar B_{\Sigma}(x_0)$ and the case where $I_\mu[e^{\varrho u}](x) = I_\mu[e^{\varrho u_{x_0, \Sigma}}](x)$ for some $x\in \bb R^n\setminus \bar B_{\Sigma}(x_0)$. If there is $x\in \bb R^n\setminus \overline B_{\Sigma}(x_0)$ for which $u(x) = u_{x_0, \Sigma}(x)$ then for any such $x$, from \eqref{eq:both_soft_signs} and Lemma \ref{lemma:u_difference} with $V = |\bb S^n|$ we have
\begin{equation*}
\begin{split}
	0
	& \geq \int_{\bb R^n\setminus B_\Sigma(x_0)}\mc L(x_0, \Sigma; x, y)\left(I_\mu[e^{\varrho u_{x_0, \Sigma}}](y)e^{\varrho u_{x_0, \Sigma}(y)} - I_\mu[e^{\varrho u}](y)e^{\varrho u(y)}\right)\; \d y\\
	& = \frac{u_{x_0, \Sigma}(x) - u(x)}{(n -1)!c_n}\\
	& = 0. 
\end{split}
\end{equation*}
Therefore, the following equality holds identically on $\bb R^n\setminus B_\Sigma(x_0)$: 
\begin{equation*}
\begin{split}
	0
	& = I_\mu[e^{\varrho u_{x_0, \Sigma}}]e^{\varrho u_{x_0, \Sigma}} - I_\mu[e^{\varrho u}]e^{\varrho u}\\
	& = I_\mu[e^{\varrho u_{x_0, \Sigma}}]\left(e^{\varrho u_{x_0, \Sigma}} - e^{\varrho u}\right) + \left(I_\mu[e^{\varrho u_{x_0, \Sigma}}] - I_\mu[e^{\varrho u}]\right)e^{\varrho u}. 
\end{split}
\end{equation*} 
Inequalities \eqref{eq:both_soft_signs} guarantee that both summands on the right-most side of this equality are nonpositive so we deduce that both $u\equiv u_{x_0, \Sigma}$ and $I_\mu[e^{\varrho u}]\equiv I_\mu[e^{\varrho u_{x_0, \Sigma}}]$ on $\bb R^n\setminus B_{\Sigma}(x_0)$. If there is $x\in \bb R^n\setminus \bar B_\Sigma(x_0)$ for which $I_\mu[e^{\varrho u}](x) = I_\mu[e^{\varrho u_{x_0, \Sigma}}](x)$ then for any such $x$ Lemma \ref{lemma:Imu_difference} and \eqref{eq:both_soft_signs} gives
\begin{equation*}
\begin{split}
	0
	& = \int_{\bb R^n\setminus B_\Sigma(x_0)}\mc K(x_0, \Sigma; x, y)\left(e^{\varrho u_{x_0, \Sigma}(y)} - e^{\varrho u(y)}\right)\; \d y\\
	& \leq 0, 
\end{split}
\end{equation*}
from which we find that $u_{x_0, \Sigma}(y) = u(y)$ for some (in fact all) $y\in \bb R^n\setminus \bar B_\Sigma(x_0)$. Arguing as above we conclude that both $u\equiv u_{x_0, \Sigma}$ and $I_\mu[e^{\varrho u}]\equiv I_\mu[e^{\varrho u_{x_0, \Sigma}}]$ on $\bb R^n\setminus B_{\Sigma}(x_0)$. This completes the verification of alternative \ref{item:both_always_bigger}, \ref{item:both_always_same}. 

To complete the proof of Lemma \ref{lemma:symmetric_about_symmetry_sphere} we assume, with the intent of obtaining a contradiction, that \ref{item:both_always_bigger} holds. By definition of $\Sigma$, for every $\sigma\in (\Sigma, \Sigma + 1)$ at least one of $\mc P(x_0, \sigma)$ or $\mc Q(x_0, \sigma)$ has positive measure so Lemma \ref{lemma:Qnonempty} guarantees that $\mc P(x_0, \sigma)$ has positive measure. Moreover, combining Lemma \ref{lemma:MS_start_main_estimate} with the fact that $1 + |\log|\cdot - x_0||\in L^{n/\varrho}(B_\sigma(x_0))$ and in view of Lemma \ref{lemma:bad_set_vanishing_measure} and Remark \ref{remark:on_lemma bad_set_vanishing_measure} we deduce the existence of $\hat\Sigma\in (\Sigma, \Sigma + 1)$ such that
\begin{equation*}
	\|e^{\varrho u_{x_0, \sigma}} - e^{\varrho u}\|_{L^{n/\varrho}(\mc P(x_0, \sigma))}
	\leq \frac 12 \|e^{\varrho u_{x_0, \sigma}} - e^{\varrho u}\|_{L^{n/\varrho}(\mc P(x_0, \sigma))}
\end{equation*}
whenever $\sigma\in [\Sigma, \hat \Sigma)$. This estimate implies that for any such $\hat \Sigma$ and any $\sigma\in [\Sigma, \hat\Sigma)$ there holds $|\mc P(x_0, \sigma)| = 0$, which is a contradiction. 
\end{proof}
\begin{lemma}
\label{lemma:all_radii_finite}
For every $x_0\in \bb R^n$ there holds $\Sigma(x_0)< \infty$. 
\end{lemma}
\begin{proof}
It suffices to show that there is $x_0\in \bb R^n$ for which $\Sigma(x_0)< \infty$. Indeed, suppose such an $x_0$ exists and fix any $z\in \bb R^n$. For any $\sigma\in (0, \Sigma(z))$ we have $u\geq u_{z, \sigma}$ in $\bb R^n\setminus B_\sigma(z)$ so 
\begin{equation*}
	\liminf_{|x|\to\infty}|x|^2e^{u(x)}
	\geq \liminf_{|x|\to\infty}|x|^2e^{u_{z, \sigma}(x)}
	= \sigma^2e^{u(z)}. 
\end{equation*}
On the other hand, from Lemma \ref{lemma:symmetric_about_symmetry_sphere} we have
\begin{equation*}
	\liminf_{|x|\to\infty}|x|^2e^{u(x)}
	= \liminf_{|x|\to\infty}|x|^2 e^{u_{x_0, \Sigma(x_0)}(x)}
	= \Sigma(x_0)^2e^{u(x_0)}. 
\end{equation*}
Combining the previous two equalities we obtain $\sigma^2 e^{u(z)}\leq \Sigma(x_0)^2e^{u(x_0)}$ for every $\sigma\in (0, \Sigma(z))$ from which we conclude that $\Sigma(z)< \infty$. It remains to establish the existence of $x_0\in \bb R^n$ for which $\Sigma(x_0)< \infty$. We do so by way of contradiction. If $\Sigma(x_0)= +\infty$ for all $x_0\in \bb R^n$ then for every $(x_0, \sigma)\in \bb R^n\times (0, \infty)$ we have 
\begin{equation*}
	\left(\frac \sigma{|x - x_0|}\right)^2e^{u(x^{x_0, \sigma})}
	\leq e^{u(x)}
	\quad \text{ for all }x\in \bb R^n\setminus B_\sigma(x_0). 
\end{equation*}
In this case an application of Lemma \ref{lemma:implies_const_or_infty} of the appendix guarantees that either $e^u\equiv \text{const}$ or $e^u\equiv+\infty$ both of which contradict the assumption $e^u\in L^n(\bb R^n)$. 
\end{proof}
\begin{proof}[Proof of Theorem \ref{theorem:classification}]
For any $x_0\in \bb R^n$ Lemma \ref{lemma:all_radii_finite} guarantees that $\Sigma(x_0)< \infty$ and Lemma \ref{lemma:symmetric_about_symmetry_sphere} guarantees that $u\equiv u_{x_0, \Sigma(x_0)}$ in $\bb R^n\setminus\{x_0\}$ so applying Lemma \ref{lemma:implies_bubble} with $\gamma = 2$ and $f = e^u$ gives 
\begin{equation}
\label{eq:eu_straggling_constants}
	e^{u(x)} = \frac{a}{d^2 + |x - \bar x|^2}
\end{equation}
for some $(\bar x, d)\in \bb R^n\times (0, \infty)$ and some $a>0$. 
In particular, $e^{\varrho u}$ is an extremal function for the sharp HLS inequality, see \eqref{eq:HLS_extremal}. Therefore, with $\mc H = \mc H(n, \mu)$ as in equation \eqref{eq:sharp_HLS_constant}, H\"older's inequality and the HLS inequality give
\begin{equation*}
\begin{split}
	\mc H\|e^{\varrho u}\|_{L^{n/\varrho}(\bb R^n)}^2
	& = \int_{\bb R^n}I_\mu[e^{\varrho u}]e^{\varrho u}\\
	& \leq \|I_\mu[e^{\varrho u}]\|_{L^{2n/\mu}(\bb R^n)}\|e^{\varrho u}\|_{L^{n/\varrho}(\bb R^n)}\\
	&\leq \mc H\|e^{\varrho u}\|_{L^{n/\varrho}(\bb R^n)}^2
\end{split}
\end{equation*}
and we deduce that equality must hold throughout this series of inequalities. In particular equality must hold in the application of H\"older's inequality so 
there is $c>0$ such that 
\begin{equation}
\label{eq:Imu_modulo_c}
	I_\mu[e^{\varrho u}](x)
	= c(e^{\varrho u(x)})^{\mu/(2n - \mu)}
	= c\left(\frac a{d^2 + |x - \bar x|^2}\right)^{\mu/2}. 
\end{equation}
In view of \eqref{eq:eu_straggling_constants}, \eqref{eq:Imu_modulo_c},  and Proposition \ref{prop:precise_decay} we have
\begin{equation}
\label{eq:compute_V}
\begin{split}
	|\bb S^n|
	& = \int_{\bb R^n}I_\mu[e^{\varrho u}]e^{\varrho u}\\
	& = c\int_{\bb R^n}\left(\frac a{d^2 + |x - \bar x|^2}\right)^n\; \d x\\
	& 
	= c\left(\frac a{2d}\right)^n|\bb S^n|,  
\end{split}
\end{equation}
where in the final equality we used the identity $|\bb S^n| = 2^n\int_{\bb R^n}(1 + |y|^2)^{-n}\; \d y$. Equation \eqref{eq:compute_V} gives
\begin{equation}
\label{eq:the_Imu_c}
	c = \left(\frac {2d}a\right)^n.
\end{equation} 
Bringing this back to \eqref{eq:Imu_modulo_c} and performing standard computations gives
\begin{equation*}
\begin{split}
	\|I_\mu[e^{\varrho u}]\|_{L^{2n/\mu}(\bb R^n)}
	& = \left(\frac {2d}a\right)^n\left(\int_{\bb R^n}\left(\frac a{d^2 + |x - \bar x|^2}\right)^n\; \d x\right)^{\mu/(2n)}\\
	& 
	=\left(\frac {2d}a\right)^{n - \mu/2}|\bb S^n|^{\mu/(2n)}. 
\end{split}
\end{equation*}
On the other hand, since $e^{\varrho u}$ is extremal for the sharp HLS inequality, 
\begin{equation*}
\begin{split}
	\|I_\mu[e^{\varrho u}]\|_{L^{2n/\mu}(\bb R^n)}
	& = \mc H\|e^{\varrho u}\|_{L^{n/\varrho}(\bb R^n)}\\
	& = \mc H\left(\int_{\bb R^n}\left(\frac a{d^2 + |x - \bar x|^2}\right)^n\; \d x\right)^{\varrho/n}\\
	& = \mc H\left(\frac{a}{2d}\right)^{\varrho}|\bb S^n|^{\varrho/n}. 
\end{split}
\end{equation*}
Equating the two expressions for $\|I_\mu[e^{\varrho u}]\|_{L^{2n/\mu}(\bb R^n)}$ gives
\begin{equation}
\label{eq:ad_relation}
	a 
	= 2d\left(\mc H|\bb S^n|^{1 - \frac\mu n}\right)^{-\frac 1{2\varrho}}. 
\end{equation}	
Bringing this back to \eqref{eq:eu_straggling_constants} yields the expression for $u$ asserted in equation \eqref{eq:nonlocal_bubbles}. The expression for $I_\mu[e^{\varrho u}]$ asserted in equation \eqref{eq:Imu_erhou_form} follows by combining \eqref{eq:the_Imu_c}, \eqref{eq:ad_relation} and \eqref{eq:Imu_modulo_c}.  
\end{proof}
\appendix\section{}
Lemmata \ref{lemma:implies_const_or_infty} and \ref{lemma:implies_bubble} are restatements of Lemmata 5.7 and 5.8 of \cite{Li2004}. 
\begin{lemma}
\label{lemma:implies_const_or_infty}
Let $n\geq 1$ and let $\gamma\in \bb R$. If $f:\bb R^n\to [-\infty, +\infty]$ is a function with the property that for every $(x_0, \sigma)\in \bb R^n\times (0, \infty)$ the inequality 
\begin{equation*}
	\left(\frac{\sigma}{|x- x_0|}\right)^\gamma
	f\left(x_0 + \frac{\sigma^2(x - x_0)}{|x - x_0|^2}\right)
	\leq f(x)
	\quad \text{ for all }x\in \bb R^n\setminus B_\sigma(x_0)
\end{equation*}
holds for all $x\in \bb R^n\setminus B_\sigma(x_0)$ then either $f\equiv\text{const}$ or $f\equiv +\infty$. 
\end{lemma}
\begin{lemma}
\label{lemma:implies_bubble}
Let $n\geq 1$, let $\gamma\in \bb R$ and let $f\in C^0(\bb R^n)$. If, for all $x_0\in \bb R^n$ there is $\Sigma(x_0)>0$ such that
\begin{equation*}
	\left(\frac{\Sigma(x_0)}{|x- x_0|}\right)^\gamma
	f\left(x_0 + \frac{\Sigma(x_0)^2(x - x_0)}{|x - x_0|^2}\right)
	= f(x)
	\quad \text{ for all }x\in \bb R^n\setminus \{x_0\}
\end{equation*}
then there are constants $a\geq 0$, $d>0$ and there is $\bar x\in \bb R^n$ such that
\begin{equation*}
	f(x)
	= \pm a(d^2 + |x - \bar x|^2)^{-\gamma/2}
	\quad \text{ for all }x\in \bb R^n. 
\end{equation*}
\end{lemma}
%


\input{nthOrderChoquardClassification-arxiv.bbl}
\end{document}

%% file: structure.tex
\usepackage[T1]{fontenc} 
\usepackage[utf8]{inputenc} 
\usepackage{amssymb}
\usepackage{graphicx} 
\graphicspath{{Figures/}} 
\usepackage{enumitem} 
	\setlist[enumerate]{label={\bf \arabic*.}, ref = {\bf\arabic*}, wide = 0pt} 
	\setlist[enumerate, 2]{label={\bf (\alph*)}, wide = 15pt, leftmargin = 15pt} 
\usepackage{subfig} 
\usepackage{varioref} 
\usepackage{tikz}
\usepackage{tikz-cd}
\usetikzlibrary{arrows,arrows.meta}
\usetikzlibrary{calc}
\usepackage{esint} 
\usepackage{amsaddr} 
\theoremstyle{plain} 
\newtheorem{theorem}{Theorem}
\newtheorem{lemma}[theorem]{Lemma}
\newtheorem{prop}[theorem]{Proposition}
\newtheorem{coro}[theorem]{Corollary}

\newtheorem{oldtheorem}{Theorem}

\theoremstyle{definition} 
\newtheorem{defn}[theorem]{Definition}

\theoremstyle{remark} 
\newtheorem{remark}[theorem]{Remark}

\newtheorem{question}[theorem]{Question}

\numberwithin{equation}{section}
\numberwithin{theorem}{section}
\usepackage[bookmarks, colorlinks=true, pdfstartview=FitV, linkcolor=blue, citecolor=blue, urlcolor=blue]{hyperref}

\newcommand{\ds}{\displaystyle}
\renewcommand{\d}{\mathrm d}

\newcommand{\bdy}{\partial}
\newcommand{\lap}{\Delta}
\newcommand{\Grad}{\nabla}

\newcommand{\mc}{\mathcal}
\newcommand{\bb}{\mathbb}
\newcommand{\abs}[1]{\left|#1\right|}

\newcommand{\lb}{\langle}
\newcommand{\rb}{\rangle}

\DeclareMathOperator{\loc}{loc}


%% file: nthOrderChoquardClassification-arxiv.bbl
\begin{thebibliography}{JMMX15}

\bibitem[ARS06]{AdimurthiRobertStruwe2006}
Adimurthi, Fr\'ed\'eric Robert, and Michael Struwe.
\newblock {Concentration phenomena for Liouville's equation in dimension four}.
\newblock {\em J. Eur. Math. Soc.}, 8(2):171--180, 2006.

\bibitem[BM91]{BrezisMerle1991}
Ha{\"\i}m Brezis and Frank Merle.
\newblock {Uniform estimates and blow--up behavior for solutions of $-\Delta u
  = V(x)e^u$ in two dimensions}.
\newblock {\em Communications in Partial Differential Equations},
  16(8-9):1223--1253, 1991.

\bibitem[CC01]{ChangChen2001}
Sun-Yung~Alice Chang and Wenxiong Chen.
\newblock A note on a class of higher order conformally covariant equations.
\newblock {\em Discrete and Continuous Dynamical systems}, 7(2):275--282, 2001.

\bibitem[CL91]{ChenLi1991}
Wenxiong Chen and Congming Li.
\newblock Classification of solutions of some nonlinear elliptic equations.
\newblock {\em Duke Mathematical Journal}, 63(3):615--622, 1991.

\bibitem[Glu20]{Gluck2020classification}
Mathew Gluck.
\newblock Classification of solutions to a system of {$n^{\text{th}}$} order
  equations on {$\mathbb R^n$}.
\newblock {\em Communications on Pure \& Applied Analysis}, 19(12), 2020.

\bibitem[Glu25a]{Gluck2025classification}
Mathew Gluck.
\newblock {Classification of solutions to an elliptic equation on $\mathbb R^2$
  with nonlocal nonlinearity}.
\newblock {\em Discrete Contin. Dyn. Syst.}, 2025.

\bibitem[Glu25b]{Gluck2025quantization}
Mathew Gluck.
\newblock Quantization for sequences of blow-up solutions to an elliptic
  equation having nonlocal exponential nonlinearity, 2025.

\bibitem[Gor61]{Gorin1961}
Evgenii~A Gorin.
\newblock Asymptotic properties of polynomials and algebraic functions of
  several variables.
\newblock {\em Russian mathematical surveys}, 16(1):93--119, 1961.

\bibitem[HM15]{HyderMartinazzi2015}
Ali Hyder and Luca Martinazzi.
\newblock {Conformal metrics on $\mathbb R^{2m}$ with constant $Q$-curvature,
  prescribed volume and asymptotic behavior}.
\newblock {\em Discrete and continuous dynamical systems Series A},
  35(1):283--299, 2015.

\bibitem[HN23]{HuangNiu2023}
Genggeng Huang and Yating Niu.
\newblock {Classification of solutions of higher order critical Choquard
  equation}.
\newblock {\em arXiv preprint arXiv:2310.08264}, 2023.

\bibitem[HY15]{HuangYe2015}
Xia Huang and Dong Ye.
\newblock {Conformal metrics in $\mathbb R^{2m}$ with constant $Q$-curvature
  and arbitrary volume}.
\newblock {\em Calculus of Variations and Partial Differential Equations},
  54:3373--3384, 2015.

\bibitem[Hyd16]{Hyder2016}
Ali Hyder.
\newblock {Existence of entire solutions to a fractional Liouville equation in
  $\mathbb R^n$}.
\newblock {\em Rendiconti Lincei}, 27(1):1--14, 2016.

\bibitem[Hyd17]{Hyder2017}
Ali Hyder.
\newblock {Conformally Euclidean metrics on $\mathbb R^n$ with arbitrary total
  $Q$-curvature}.
\newblock {\em Analysis \& PDE}, 10(3):635--652, 2017.

\bibitem[Hyd19]{Hyder2019}
Ali Hyder.
\newblock {Structure of conformal metrics on $\mathbb R^n$ with constant
  $Q$-curvature}.
\newblock {\em Differ. Integral Equ.}, 32:423--454, 2019.

\bibitem[JMMX15]{Jin2015}
Tianling Jin, Ali Maalaoui, Luca Martinazzi, and Jingang Xiong.
\newblock Existence and asymptotics for solutions of a non-local
  {$Q$}-curvature equation in dimension three.
\newblock {\em Calc. Var.}, 2015(52):469--488, 2015.

\bibitem[Li04]{Li2004}
YanYan Li.
\newblock Remark on some conformally invariant integral equations: the method
  of moving spheres.
\newblock {\em Journal of the European Mathematical Society}, 6(2):153--180,
  2004.

\bibitem[Lie83]{Lieb1983}
Elliott~H. Lieb.
\newblock {Sharp constants in the Hardy-Littlewood-Sobolev and related
  inequalities}.
\newblock {\em Annals of Mathematics}, 118(2):349--374, 1983.

\bibitem[Lin98]{Lin1998}
C.~S. Lin.
\newblock {A classification of solutions of a conformally invariant fourth
  order equation in $\mathbb R^n$}.
\newblock {\em Commentarii Mathematici Helvetici}, 73:206--231, 1998.

\bibitem[LS94]{LiShafrir1994}
Yanyan Li and Itai Shafrir.
\newblock {Blow-up analysis for solutions of $-\Delta u = Ve^u$ in dimension
  two}.
\newblock {\em Indiana University Mathematics Journal}, 43(4):1255--1270, 1994.

\bibitem[Mal06]{Malchiodi2006}
Andrea Malchiodi.
\newblock Compactness of solutions to some geometric fourth-order equations.
\newblock {\em Journal für die reine und angewandte Mathematik},
  2006(594):137--174, 2006.

\bibitem[Mar09a]{Martinazzi2009}
Luca Martinazzi.
\newblock {Classification of solutions to the higher order Liouville’s
  equation on $\mathbb R^{2m}$}.
\newblock {\em Mathematische Zeitschrift}, 263(2):307--329, 2009.

\bibitem[Mar09b]{Martinazzi2009concentration}
Luca Martinazzi.
\newblock Concentration–compactness phenomena in the higher order liouville's
  equation.
\newblock {\em Journal of Functional Analysis}, 256(11):3743--3771, 2009.

\bibitem[Mar11]{Martinazzi2011}
Luca Martinazzi.
\newblock {Quantization for the prescribed $Q$-curvature equation on open
  domains}.
\newblock {\em Communications in Contemporary Mathematics}, 13(03):533--551,
  2011.

\bibitem[Mar13]{Martinazzi2013}
Luca Martinazzi.
\newblock {Conformal metrics on $\mathbb R^{2m}$ with constant $Q$-curvature
  and large volume}.
\newblock {\em Annales de l'IHP Analyse non lin{\'e}aire}, 30(6):969--982,
  2013.

\bibitem[MP10]{MartinazziPetrache2010}
Luca Martinazzi and Mircea Petrache.
\newblock Asymptotics and quantization for a mean-field equation of higher
  order.
\newblock {\em Communications in Partial Differential Equations},
  35(3):443--464, 2010.

\bibitem[Rob07]{Robert2007}
Frédéric Robert.
\newblock Quantization effects for a fourth-order equation of exponential
  growth in dimension $4$.
\newblock {\em Proceedings of the Royal Society of Edinburgh: Section A
  Mathematics}, 137(3):531–553, 2007.

\bibitem[RS04]{RobertStruwe2004}
Fr{\'e}d{\'e}ric Robert and Michael Struwe.
\newblock Asymptotic profile for a fourth order pde with critical exponential
  growth in dimension four.
\newblock {\em Advanced Nonlinear Studies}, 4(4):397--415, 2004.

\bibitem[WY08]{WeiYe2008}
Juncheng Wei and Dong Ye.
\newblock {Nonradial solutions for a conformally invariant fourth order
  equation in $\mathbb R^4$}.
\newblock {\em Calculus of Variations \& Partial Differential Equations},
  32(3), 2008.

\end{thebibliography}
